\documentclass[11pt,reqno]{amsart}

\usepackage[final]{graphicx}
\usepackage{amsfonts}
\usepackage{amsmath}
\usepackage{amssymb}
\usepackage{amsthm}
\usepackage{mathrsfs}
\usepackage[colorlinks=true,linkcolor=blue,citecolor=blue]{hyperref}
\usepackage{bbm}
\numberwithin{equation}{section}
\newtheorem{theorem}{Theorem}[section]
\newtheorem{lemma}[theorem]{Lemma}

\newtheorem{proposition}[theorem]{Proposition}
\newtheorem{remark}[theorem]{Remark}

\makeatletter\@namedef{subjclassname@2020}{\textup{2020} Mathematics Subject Classification}

\DeclareMathOperator{\im}{Im}

\newcommand{\R}{{\mathord{\mathbb R}}}
\newcommand{\C}{{\mathord{\mathbb C}}}

\newcommand{\N}{{\mathord{\mathbb N}}}

\newcommand{\Sph}{\mathbb{S}}

\renewcommand{\(}{\left(}
\renewcommand{\)}{\right)}

\newcommand{\be}[1]{\begin{equation}\label{#1}}
\newcommand{\ee}{\end{equation}}

\begin{document}

\title[Stability of Sobolev inequalitieson the Heisenberg group]{Asymptotically sharp stability of Sobolev inequalities on the Heisenberg group with  dimension-dependent constants}

\author{Lu Chen}
\address[Lu Chen]{Key Laboratory of Algebraic Lie Theory and Analysis of Ministry of Education, School of Mathematics and Statistics, Beijing Institute of Technology, Beijing
100081, PR China; MSU-BIT-SMBU Joint Research Center of Applied Mathematics, Shenzhen MSU-BIT
University, Shenzhen 518172, China}
\email{chenlu5818804@163.com}

\author{Guozhen Lu}
\address[Guozhen Lu]{Department of Mathematics, University of Connecticut, Storrs, CT 06269, USA}
\email{guozhen.lu@uconn.edu}

\author{Hanli Tang}
\address[Hanli Tang]{Laboratory of Mathematics and Complex Systems (Ministry of Education), School of Mathematical Sciences, Beijing Normal University, Beijing, 100875, China}
\email{hltang@bnu.edu.cn}

\author{Bohan Wang}
\address[Bohan Wang]{Key Laboratory of Algebraic Lie Theory and Analysis of Ministry of Education, School of Mathematics and Statistics, Beijing Institute of Technology, Beijing
100081, PR China}
\email{}

\address{}
\thanks{The first author was partly supported by the National Key Research and Development Program (No.
2022YFA1006900) and National Natural Science Foundation of China (No. 12271027). The second author was partly supported by a Simons grant and a Simons Fellowship from the Simons Foundation.  The third author
was partly supported by National Natural Science Foundation of China (Grant No.12471100), the Fundamental Research Funds for the Central Universities(Grant No.2233300008)}


\begin{abstract}
In this paper, we are concerned with the optimal asymptotic lower bound for the stability of Sobolev inequality on the Heisenberg group. We first establish the optimal local stability of Sobolev inequality on the  CR sphere through bispherical harmonics and complicated orthogonality technique ( see Lemma 3.1). The loss of the P\"{o}lya-Szeg\"{o} inequality and the Riesz rearrangement inequality in the CR setting makes it impossible to use any rearrangement flow technique (either differential rearrangement flow or integral rearrangement flow) to derive the optimal stability of Sobolev inequality on the CR sphere from corresponding optimal local stability. To circumvent this, we will use the CR Yamabe flow to  establish the optimal stability of Sobolev inequality on the Heisenberg group with the dimension-dependent constants (see Theorem 1.1). As an application, we also establish the optimal stability of the Hardy-Littlewood-Sobolev  (HLS) inequality for special conformal index with the dimension-dependent constants (see Theorem 1.3).
Our approach is rearrangement-free and can be used to study the optimal stability problem for fractional Sobolev inequality or HLS inequality on the Heisenberg group once the corresponding continuous flow is established.
\end{abstract}

\maketitle

\section{Introduction and main results}\label{sec:intro}
We shall be concerned with the optimal stability of Sobolev inequalities on the Heisenberg group with dimension-dependent constants. The analogues on $\mathbb{R}^n$
have been established by Dolbeault, Esteban, Figalli, Frank and Loss \cite{DEFFL} for the first order Sobolev inequality and by Chen, Lu and Tang \cite{CLT1}, \cite{CLT2}, \cite{CLT3} for the higher  order and arbitrary fractional order Sobolev inequalities.\vskip0.1cm

The stability problem of Sobolev inequality was posed  by Brezis and Lieb. In \cite{BrLi} they asked if the following refined first order Sobolev inequality with a reminder term on $\mathbb{R}^n$
holds for some distance function $d$:
$$\left\|\nabla U \right\|_2^2 - \mathcal S_{1,n} \| U\|_{\frac{2n}{n-2}}^2\geq c d^{2}(U, \mathcal{M}_1),$$
where $\mathcal{M}_1$ is the set of extremal functions of the sharp first order Sobolev inequality. This question was answered affirmatively  in a pioneering work by Bianchi and Egnell \cite{BE}.  Stabilities of  fractional Sobolev inequality for $s=2$,  even positive integer $s<n/2$
and all fractional $0<s<n/2$  were established by the second author and
Wei \cite{LuWe}, by Bartsch, Weth and Willem \cite{BaWeWi}, and by Chen, Frank and Weth \cite{ChFrWe}
respectivelly. In fact, Chen, Frank and Weth  proved that there exits $C_{n,s}>0$ such that
\begin{equation}\label{Sob sta ine}
\left\|(-\Delta)^{s/2} U \right\|_2^2 - \mathcal S_{s,n} \| U\|_{\frac{2n}{n-2s}}^2\geq C_{n,s} d^{2}(U, \mathcal{M}_s),
\end{equation}
for all $U\in \dot H^s(\mathbb{R}^n)$, where $d(U,\mathcal{M}_s)=\min\{\|(-\Delta)^{s/2}(U-\phi)\|_{L^2}:\phi \in \mathcal{M}_s\}$ with $\mathcal{M}_s$ being the set of extremal functions of Sobolev inequalities and $\mathcal S_{s,n}$ being the sharp constant.

Initially, in these works  \cite{LuWe},  \cite{BaWeWi} and  \cite{ChFrWe}, the stability of Sobolev inequality was proved by establishing the local stability
of Sobolev inequality based on the spectrum analysis of elliptic or high order elliptic operator and using the Lions' concentration compactness technique to obtain global stability of Sobolev inequality. However, this method
does not tell us more information about the constant $C_{n,s}$ except that it is positive.
Recently, in the  work of Dolbeault, Esteban, Figalli, Frank and Loss in \cite{DEFFL}, they made the first important breakthrough. They proved

\vskip0.3cm
\textbf{Theorem A. }
\textit{There is an explicit constant} $\beta>0$ \textit{such that for all} $n\geq 3$ \textit{and for all} $U\in \dot H^1(\mathbb{R}^n)$, \textit{there holds}
$$\left\|\nabla U \right\|_2^2 - \mathcal S_{1,n} \| U\|_{\frac{2n}{n-2}}^2\geq \frac{\beta}{n}d^2(U,\mathcal{M}_1).$$
\vskip0.3cm

In  \cite{DEFFL}, they reduced the proof of the global stability to the local stability of the first order Sobolev inequality based on the competing symmetries and the continuous Steiner symmetrization inequality (rearrangement inequality), and at the same time they provided new  techniques to obtain  a refined  quantitative local stability. Furthermore, Chen, Lu and Tang in \cite{CLT1} developed the integral rearrangement flow technique and the refined  duality stability method, which was initially developed by Carlen \cite{Car1},  to establish the stability of Sobolev inequality for the arbitrary order with the explicit lower bound. However, the explicit lower bound was not  optimal there.
In \cite{CLT2}, the authors further established the optimal stability of fractional Sobolev inequality ($s\in (0,1)$) with the dimension-dependent and order-dependent constants using the Aronszajn-Smith formula and an important deficit-split inequality. However, the optimal stability of Sobolev inequality of  arbitrary order (particularly $1<s<\frac{n}{2}$) was not solved. In the more recent paper \cite{CLT3}, Chen, Lu and Tang proposed the weak-decomposition-strong-stability technique and then applied improved integral rearrangement flow technique to directly establish the optimal stability of HLS inequality without using any duality argument. Applying the dual stability into HLS inequality improved in \cite{CLT1}, they obtained the optimal stability of Sobolev inequality for the arbitrary order. In summary, the results of \cite{CLT2, CLT3} about the optimal stability of fractional Sobolev inequality can be stated as follows:

\vskip0.3cm
\textbf{Theorem B }
\textit{For any fixed} $s\in(0,n/2)$, \textit{there exist a positive constant} $\beta_s$ \textit{such that for any} $f\in  \dot H^s(\mathbb{R}^n)$, \textit{there holds}
$$\left\| (-\Delta)^{s/2} f \right\|_2^2-\mathcal S_{s,n} \|f\|_{\frac{2n}{n-2s}}^2\geq \frac{\beta_{s}}{n} \inf_{h\in\mathcal{ M}_s}\|(-\Delta)^{s/2}(f-h)\|_2^2.$$
Furthermore, when $s\rightarrow 0$, $\beta_s=O(s)$.
\vskip0.3cm

We also mention that
Dolbeault, Esteban, Figalli, Frank and Loss \cite{DEFFL} established the optimal optimal stability of Gross-type Log-Sobolev inequality in $\mathbb{R}^n$ \cite{Gross}  through dimension-limit method as the dimension goes to infinity and Chen, Lu, Tang \cite{CLT2} established the optimal stability of Beckner's Log-Sobolev inequality
 in $\mathbb{S}^n$ obtained  in \cite{Be1992} through order-limit method as the order goes to zero.
This sharpens the local stability of Log-Sobolev inequality in $\mathbb{S}^n$ obtained earlier in \cite{CLT}.

\medskip

However, all of the works by  Dolbeault, Esteban, Figalli, Frank and Loss \cite{DEFFL} and by Chen, Lu and Tang \cite{CLT1, CLT2, CLT3} depend on the rearrangement inequality on $\mathbb{R}^n$. More precisely, in \cite{DEFFL} the Steiner symmetrization inequality (Poly\'{a}-Szeg\"{o})  for the $L^2$ integral of the gradient of $u$ is crucial while in \cite{CLT2}, \cite{CLT3} the Riesz rearrangement inequality for the HLS integral is also important. As we know, these rearrangement inequalities are absent on the Heisenberg group, which is the main difficulty we need to overcome in the paper.
\medskip

The Heisenberg group $\mathbb{H}^n$ is $\C^n\times\R$ with elements $u=(z,t)$ and group law
$$
u u' = (z,t) (z',t') = (z+z', t+t'+2\im z\cdot\overline{z'}) \,.
$$
Here we have set $z\cdot\overline{z'} = \sum_{j=1}^n z_j \overline{z_j'}$. Haar measure on $\mathbb{H}^n$ is the usual Lebesgue measure $du=dz\,dt$. (To be more precise, $dz=dx\,dy$ if $z=x+iy$ with $x,y\in\R^n$.) We write $\delta u = (\delta z,\delta^2 t)$ for dilations of a point $u=(z,t)$ and denote the homogeneous norm on $\mathbb{H}^n$ by
$$
| u | = |(z,t)| = (|z|^4+t^2)^{1/4} \,.
$$
As usual, we denote the homogeneous dimension by $Q:=2n+2$.
The canonical subelliptic Laplacian on $\mathbb{H}^n$ is the second-order differential operator defined as
$$\Delta_{\mathbb{H}^n}=\sum\limits_{j=1}^n(X_j^2+Y^2_j),$$
where
\[X_j=\frac{\partial}{\partial x_j}+2y_j\frac{\partial}{\partial t}, \quad Y_j=\frac{\partial}{\partial y_j}-2x_j\frac{\partial}{\partial t},
\]
for $j=\{1,...,n\}$. Denote $|\nabla_{\mathbb{H}^n}f|$  the norm of the subelliptic gradient of the function $f:\mathbb{H}^n\rightarrow \mathbb{R}:$
$$|\nabla_{\mathbb{H}^n}f|=\left(\sum_{j=1}^n(X_jf)^2+(Y_jf)^2\right)^{1/2}.$$ We also denote the Sobolev (or folland-Stein) space $S^{1}(\mathbb{H}^n)$ by the completion of $C_0^\infty(\mathbb{H}^n)$ under the norm $\|\nabla_{\mathbb{H}^n}f\|_2$.
\vskip0.1cm

Sobolev inequalities on the Heisenberg group was first studied by Folland and Stein \cite{FS}, while Jerison and Lee \cite{JeLe} proved the following celebrated sharp Sobolev inequality on the Heisenberg group.

\vskip0.3cm
\textbf{Theorem B. }
\textit{The Sobolev inequality on the Heisenberg group states}
\begin{equation}
 \label{sobol1}
 \int_{\mathbb{H}^{n}} |\nabla_{\mathbb{H}^n}f|^2 dzdt
\geq \frac{4\pi n^2 }{(2^{2n} n!)^{1/(n+1)} } \|f\|_{\frac{2Q}{Q-2}}^2,
\end{equation}
  \textit{where the constant} $\frac{4\pi n^2 }{(2^{2n} n!)^{1/(n+1)} }$ \textit{is sharp. The equality holds if and only if}
$$f(z,t)=K|w+z\cdot \mu+\lambda|^{-n}$$
\textit{where} $w=t+i|z|^2$, $K,\lambda\in \mathbb{C}$, ${\rm Im} \lambda >\frac{|\mu|^2}{4}$ \textit{and} $\mu\in\mathbb{C}^n$.

\vskip0.3cm
The Cayley transform on $\mathbb{H}^n$, a generalization of stereographic transform on Euclidean space $\mathbb{R}^n$, $\mathcal{C}: \mathbb{H}^n\rightarrow \mathbb{S}^{2n+1} \setminus \{O\}$ with $O$ being the south pole $(0,\ldots,0,-1)$, is defined by
 \[u=(z,t)\mapsto \xi=(\xi',\xi_{n+1})=\left(\frac{2z}{1+|z|^2-it},\frac{1-|z|^2+it}{1+|z|^2-it}\right),\]
  and its inverse $\mathcal{C}^{-1}:\mathbb{S}^{2n+1}\rightarrow \mathbb{H}^{n} $ is given by
  \[\mathcal{C}^{-1}(\xi)=(\xi',\xi_{n+1})=\left(\frac{\xi_1}{1+\xi_{n+1}},...,\frac{\xi_{n}}{1+\xi_{n+1}},{\rm Im}\frac{1-\xi_{n+1}}{1+\xi_{n+1}}\right),\]
   where $ \mathbb{S}^{2n+1}:=\{\xi=(\xi_1,...,\xi_{n+1}): \sum\limits_{i=1}^{n+1}|\xi_i|^2=1 \}$ can be considered as
    the complex sphere of  $\mathbb C^{n+1}$ (CR sphere). The Jacobian of the Cayley transform is
 \[|J_\mathcal{C}|=2^{Q-1}((1+|z|^2)^2+|t|^2)^{-Q/2}=2^{-1}|1+\xi_{n+1}|^Q.\]
Via the Cayley transform, there is an equivalent version of Sobolev inequality on CR sphere $\mathbb{S}^{2n+1}$. In order to state this inequality, let us introduce some notations first.
\vskip0.1cm

For $j=1,\ldots,n+1$, define the operators
$$
T_j := \frac{\partial}{\partial \xi_j}-\overline{\xi_j} \sum_{k=1}^{n+1} \xi_k \frac{\partial}{\partial \xi_k}  \,,
\qquad
\overline{T_j} := \frac{\partial}{\partial \overline{\xi_j}}-\xi_j \sum_{k=1}^{n+1} \overline{\xi_k} \frac{\partial}{\partial \overline{\xi_k}} \,,
$$
The conformal sublaplacian is defined to be
$$
\mathcal L := -\frac12 \sum_{j=1}^{n+1} \left(\overline{T_j}T_j + T_j \overline{T_j}\right) + \frac{n^2}4 =\mathcal{L}'+\frac{n^2}{4}.
$$
The associated quadratic form is
\begin{equation}
 \label{eq:esph}
\mathcal E[u] := \frac12 \int_{\Sph^{2n+1}} \left( \sum_{j=1}^{n+1} \left( |T_j u|^2 + |\overline{T_j} u|^2 \right) + \frac{n^2}2 |u|^2 \right) \,d\xi \,,
\end{equation}
where $d\xi$ denotes the surface measure on the sphere $\mathbb{S}^{2n+1}$.
We define the Sobolev space $S^1(\Sph^{2n+1}):=\{u(\xi):\mathcal E[u(\xi)]<+\infty,~~~~~~\xi\in\mathbb{S}^{2n+1} \}$.  By the Cayley transform, the Sobolev inequality \eqref{sobol1} on the Heisenberg group $\mathbb{H}^n$ is equivalent to the Sobolev inequality on the CR sphere $\mathbb{S}^{2n+1}$, i.e.
\begin{equation}
 \label{eqsobq}
\mathcal E[u]
\geq (\frac{Q-2}{4})^2|\mathbb{S}^{2n+1}|^{2/Q} \left( \int_{\Sph^{2n+1}} |u|^{\frac{2Q}{Q-2}} \,d\xi
\right)^{\frac{Q-2}{Q}} ,
\end{equation}
for any $u\in S^1(\Sph^{2n+1})$.
Furthermore, the equality holds if and only if
$$u(\xi)=\frac{c}{|1-\overline{\eta}\cdot\xi|^{\frac{Q-2}{2}}}$$
for some $c\in \mathbb{C}$, $\eta=(\eta_1,...,\eta_{n+1})\in \mathbb{C}^{n+1}$ with $|\eta| <1$.
\vskip0.3cm

The sharp Sobolev inequality on CR sphere $\mathbb{S}^{2n+1}$ was also obtained by Frank and Lieb in \cite{Frank-Lieb} by a different proof from Jerison and Lee's result in \cite{JeLe} . In the same paper, Frank and Lieb also obtained the sharp fractional Sobolev inequality on the Heisenberg group by establishing the sharp Hardy-Littlewood-Sobolev inequality on the Heisenberg group and using a duality argument. In fact, Frank and Lieb proved in their remarkable work \cite{Frank-Lieb}

\vskip0.3cm
\textbf{Theorem C. }
\textit{Let~~}$0<\lambda<Q$, \text{then for any} $f\in L^{\frac{2Q}{Q-\lambda}}(\mathbb{H}^n)$
\begin{equation}
\left|\iint_{\mathbb{H}^n\times \mathbb{H}^n}\frac{\bar{f}(u)f(v)}{|u^{-1}v|^\lambda}du dv \right|\leq \left(\frac{\pi^{n+1}}{2^{n-1}n!}\right)^{\lambda/Q}\frac{n!\Gamma(\frac{Q-\lambda}{2})}{\Gamma^2(\frac{2Q-\lambda}{4}
)}\|f\|^2_{L^{\frac{2Q}{2Q-\lambda}}(\mathbb{H}^n)},
\end{equation}
 \textit{with equality holds if and only if}
$$f(u)=cF(\delta(a^{-1}u))$$
\textit{for some} $c\in \mathbb{C},~~\delta>0$, $a\in\mathbb{H}^n$ \textit{and} $F$ \textit{is the function with the form}
$$F(u)=((1+|z|^2)^2+t^2)^{-(2Q-\lambda)/4}.$$
\vskip0.3cm
We also mention that Beckner \cite{Be1997} obtained the sharp Stein-Weiss inequality with the horizontal weight $|z|$ in some special index by identifying the best constants and extremal functions.
Han, Lu and Zhu \cite{HLZ} established the Stein-Weiss inequality on the Heisenberg group with both of full weight $|(z,t)|$ and horizontal weight $|z|$, while
Chen, Lu and Tao \cite{CLTao} established the corresponding existence of extremals using double-weighted Lions concentration-compactness principle.
\medskip

The bispherical harmonics is an important tool to study the Sobolev and Hardy-Littlewood-Sobolev inequalities on the CR sphere $\mathbb{S}^{2n+1}$. Let us denote $\mathcal H_{j,k}$ the space of restrictions to $\Sph^{2n+1}$ of harmonic polynomials $p(z,\overline z)$ on $\C^{n+1}$ which are homogeneous of degree $j$ in $z$ and degree $k$ in $\overline z$; see \cite{Fo} and references therein. It is well known that the  conformal sublaplacian $\mathcal{L}$ acts as multiplication on  $\mathcal H_{j,k}$ by

\begin{equation}\label{eig}
\lambda_{j,k}= \frac{\Gamma(j+\frac{Q+2}{4})\Gamma(k+\frac{Q+2}{4})}{\Gamma(j+\frac{Q-2}{4})\Gamma(k+\frac{Q-2}{4})}= (\frac{Q-2}{4}+k)(\frac{Q-2}{4}+j).
\end{equation}

The stability of Sobolev inequalities on the Heisenberg group has been established by  Loiudice \cite{Lo}, by considering the  equivalence between  eigenvalue problems on the Heisenberg group and the CR sphere as observed in Malchiodi and Uguzzoni \cite{MU}. Loiudice proved in \cite{Lo} that there exists a  positive constant $C_Q$ such that
\begin{equation}\label{sob sta H}
\int_{\mathbb{H}^{n}} |\nabla_{\mathbb{H}^n}f|^2 dzdt
-\frac{4\pi n^2 }{(2^{2n} n!)^{1/(n+1)} } \|f\|_{\frac{2Q}{Q-2}}^2\geq C_Q d^2(f,\mathcal{M}),
\end{equation}
where
$$d^2(f,\mathcal{M})=\inf_{h\in \mathcal{M}}\|\nabla_{\mathbb{H}^n}(f-h)\|^2_2$$
and $\mathcal{M}$ is the real-valued extremal functions of the Sobolev inequality on the Heisenberg group, i.e.
$$\mathcal{M}=\{cH(\delta(a^{-1}u)), c\in \mathbb{R},~~\delta>0,~~a\in \mathbb{H}^n\}$$
with $H(u)=\left((1+|z|^2)^2+t^2\right)^{-\frac{Q-2}{4}}$. Later, Liu and Zhang \cite{LiZh} proved the the stability of fractional Sobolev inequalities on the Heisenberg group. The strategy adopted in \cite{Lo}  is similar to that used in \cite{BE} while the  method used in \cite{LiZh} is similar to that in \cite{ChFrWe}. Nevertheless, we still know nothing about the constant $C_Q$ except knowing that it is positive.
\medskip

The main goal of the paper is to establish the following stability of Sobolev inequalities on the Heisenberg group with optimal dimension-dependent constants for real-valued functions.

\begin{theorem}\label{sob stability}
There exists a constant $\beta_0>0$ independent of $Q$ such that for any $f\in S^1(\mathbb{H}^{n})$, there holds
\begin{equation*}
\int_{\mathbb{H}^{n}} |\nabla_{\mathbb{H}^n}f|^2 dzdt
-\frac{4\pi n^2 }{(2^{2n} n!)^{1/(n+1)} } \|f\|_{\frac{2Q}{Q-2}}^2\geq \frac{\beta_0}{Q} d(f,\mathcal{M})^{2}.
\end{equation*}
\end{theorem}

\begin{remark}
 From Liu and Zhang's work \cite{LiZh} we know that the best constant of stability of Sobolev inequalities (\ref{sob sta H}) satisfies
$$C_Q\leq \frac{4}{Q+6}.$$
So $1/Q$ is optimal as $Q\rightarrow \infty$.
\end{remark}

Using the stability of Sobolev inequalities on the Heisenberg group and the dual stability method established by Carlen (see \cite{Car1}, \cite{Car2}), we can obtain the stability of Hardy-Littlewood-Sobolev inequalities
on the Heisenberg group for the specific index $\lambda=Q-2$.

Denote by $S_{Q}=\left(\frac{2 \pi^{n+1}}{n!}\right)^{\frac{Q-2}{Q}}\frac{n!}{\Gamma^2((Q+2)/4)}$ the best constant of the HLS inequality and denote by $\mathcal{M}_{HLS}$ the real-valued extremal functions of the HLS inequality on the Heisenberg group when $\lambda=Q-2$, i.e.
$$\mathcal{M}_{HLS}=\{cF(\delta(a^{-1}u)), c\in \mathbb{R},~~\delta>0,~~a\in \mathbb{H}^n\}$$
with $F(u)=\left((1+|z|^2)^2+t^2\right)^{-\frac{Q+2}{4}}$. Then the following stability of HLS inequalities on the Heisenberg group with the optimal dimension-dependent constants holds.

\begin{theorem}\label{HLS stability}
There exists a positive $\alpha$ independent of $Q$ such that
\begin{equation*}
\|f\|^2_{L^{\frac{2Q}{Q+2}}(\mathbb{H}^n)}-S_Q^{-1}\iint_{\mathbb{H}^n\times \mathbb{H}^n}\frac{f(u)f(v)}{|u^{-1}v|^{Q-2}}du dv \geq \frac{\alpha}{Q} d^2(f,\mathcal{M}_{HLS}),
\end{equation*}
where
$$d(f,\mathcal{M}_{HLS})=\inf_{h\in \mathcal{M}_{HLS}}\|f-h\|_{L^{\frac{2Q}{Q+2}}(\mathbb{H}^n)}.$$
\end{theorem}
\begin{remark}
We would like to point out that the lower bound for the stability of the
Hardy-Littlewood-Sobolev inequality on the Heisenberg group in Theorem \ref{HLS stability} is asymptotically optimal when $Q\rightarrow \infty$ for the same reason as the analogues
on $\mathbb{R}^n$ (See Remark 1.5 in \cite{CLT2}).
\end{remark}

Let us give a brief overview over the main ideas of the proof of Theorem 1.1. By Cayley transform, we only need to establish the following stability of the Sobolev inequality on CR sphere $\mathbb{S}^{2n+1}$ with the optimal asymptotic lower bound
\begin{equation*}
 \label{stability}
\mathcal E[u]
-(\frac{Q-2}{4})^2|\mathbb{S}^{2n+1}|^{\frac{2}{Q}}  \left( \int_{\Sph^{2n+1}} |u|^{\frac{2Q}{Q-2}} \,d\xi \right)^{\frac{Q-2}{Q}}\geq \frac{\beta_0}{Q}\inf_{h\in \mathcal{M}_\ast} \mathcal E[u-h],
\end{equation*}
where $$\mathcal{M}_{*}=\left\{\frac{c}{|1-\overline{\eta}\cdot\xi|^{\frac{Q-2}{2}}}:~~c\in \mathbb{R},~~\eta=(\eta_1,...,\eta_{n+1})\in \mathbb{C}^{n+1}\text{with}~~|\eta| <1\right\}.$$
We first set up the local stability of the Sobolev inequality in $\mathbb{S}^{2n+1}$ with the optimal asymptotic lower bound, where we use the techniques introduced  in \cite{DEFFL}. They use `cutting' at various height and estimates by spherical harmonics on sphere $\mathbb{S}^n$, while we do the similar things on CR sphere $\mathbb{S}^{2n+1}$ although  calculations here are more complex.  Then we reduce the global stability to local stability. In Dolbeault, Esteban, Figalli, Frank and Loss's work, they introduced a flow, based on competing symmetries of Carlen and Loss  \cite{CaLo} and a continuous Steiner symmetrization,
to reduce the global stability to local stability. However,  because of the loss of symmetrization  in the setting of Heisenberg group, we can't adopt their argument. Coming back to stability of Sobolev inequality on the sphere $\mathbb{S}^n$, we find that the Yamabe flow is a perfect replacement for rearrangement flow. Since it is a continuous flow preserving Sobolev energy decreasing and volume unchanged, it is exactly the flow what we need to help to reduce the global stability to local stability. Furthermore, unlike rearrangement flow which was used in \cite{DEFFL}, and as well in \cite{CLT1, CLT2, CLT3},  only keeping the Sobolev-energy continuously changed from the right, this kind of flow can keep the Sobolev energy continuously decreasing, hence avoiding some complicated discussions and calculations in the proof of global stability of Sobolev inequality. For the setting of CR sphere, we expect CR Yamabe flow can help to establish the quantitative relation between the global stability and local stability for Sobolev inequality on CR sphere $\mathbb{S}^{2n+1}$. Recall that CR Yamabe flow was first introduced by Jerison and Lee in \cite{JeLe1} in order to solving CR Yamabe problem on CR manifold. In the past few decades, the existence of long time and convergence of CR Yamabe flow has been studied extensively under various assumptions. For the CR Yamabe flow on the CR sphere $\mathbb{S}^{2n+1}$, Ho in \cite{P} established the convergence of CR Yamabe flow (see \cite{P1} for an alternative proof). This makes it  likely for us to establish the optimal stability for Sobolev inequality on CR sphere.
\vskip0.1cm

This paper is organized as follows. Section 2 is devoted to prove the global stability from local stability by CR Yamabe flow on the CR sphere $\mathbb{S}^{2n+1}$. In Section 3, we
will give the local stability of Sobolev inequalities with the optimal asymptotic lower
bound. In Section 4, we will establish the stability of HLS inequality for the special conformal index on the Heisenberg group with the optimal dimension-dependent constants by the dual stability method.
\medskip

{\bf Acknowledgement.} The authors wish to thank A. Malchiodi for his helpful discussions on the CR Yamabe flow and for pointing out several useful references and R. Frank for his suggestion on only considering the global stability of Sobolev inequality  for the real-valued functions.

\section{The local stability implies the global stabilty}\label{sec:2}
 Let $\mathcal{C}$ be the Cayley transform from $\mathbb{H}^n$ to $\mathbb{S}^{2n+1}$ and $J_{\mathcal{C}}$ be the Jacobian of this transform. For any function $u$ on $\mathbb{S}^{2n+1}$, define the function $F$ on $\mathbb{H}^n$ by $F=|J_{\mathcal{C}}|^{\frac{Q-2}{2Q}}u\circ \mathcal{C}$, then
$$\|u\|_{L^{\frac{2Q}{Q-2}}(\mathbb{S}^{2n+1})}=\|F\|_{L^{\frac{2Q}{Q-2}}(\mathbb{H}^{n})},$$
and
$$\int_{\mathbb{H}^{n}} |\nabla_{\mathbb{H}^n}f|^2 dzdt=2^{2+\frac{1}{n+1}}\mathcal{E}(u),$$
which can be easily checked (see also \cite{JeLe1,Frank-Lieb}). By Cayley transform, we only need to establish the following stability of Sobolev inequality on the CR sphere $\mathbb{S}^{2n+1}$ with optimal dimension-dependent constants.
\vskip0.3cm
\begin{theorem}\label{sob sta sph}
There exists an explicit constant $\beta$ such that for any $u\in S^1(\Sph^{2n+1})$, there holds
\begin{equation}
 \label{stability}
\mathcal E[u]
-|\mathbb{S}^{2n+1}|^{\frac{Q}{2}}(\frac{Q-2}{4})^2   \left( \int_{\Sph^{2n+1}} |u|^{\frac{2Q}{Q-2}} \,d\xi \right)^{\frac{Q-2}{Q}}\geq \frac{\beta}{Q}\inf_{g\in \mathcal{M}_\ast} \mathcal E[u-g].
\end{equation}
\end{theorem}
\vskip0.1cm

Our proof is based on the following local stability of Sobolev inequality for positive functions on $\mathbb{S}^{2n+1}$ which will be proved in Section 3, Yamabe flow on $\mathbb{S}^{2n+1}$ and a concavity argument.
\begin{lemma}\label{local sta}
There exist nonnegative constants $\delta\in(0,1)$ and $c_0$ which is independent of $n$ such that for all $0\leq u\in S^1(\Sph^{2n+1})$ with
$$\inf_{g\in \mathcal{M}_\ast} \mathcal E[u-g]\leq \delta \mathcal E[u],$$
there holds
$$\mathcal E[u]
-|\mathbb{S}^{2n+1}|^{\frac{Q}{2}}(\frac{Q-2}{4})^2   \left( \int_{\Sph^{2n+1}} |u|^{\frac{2Q}{Q-2}} \,d\xi
\right)^{\frac{Q-2}{Q}}\geq \frac{c_0}{Q}\inf_{g\in \mathcal{M}_\ast} \mathcal E[u-g].$$

\end{lemma}

\medskip

\subsection{A new proof for global stability of Sobolev inequality in $\mathbb{S}^n$ through local stability}
Let us first introduce the Yamabe flow on the sphere $\mathbb{S}^n$ and give a new proof for global stability of Sobolev inequality in $\mathbb{S}^n$ without using any rearrangement technique.
\medskip

Let us denote by $\mathcal{M}_{S}$ the extremal set of Sobolev inequality on the sphere $\mathbb{S} ^n$, that is
$$\mathcal{M}_{S}=\left\{c\left(\frac{\sqrt{1-|\xi|^2}}{1-\xi\cdot w}\right)^{\frac{n-2}{2}}:\, ~\xi\in \mathbb{R}^{n+1}, c\in \mathbb{R}, |\xi|<1\right\}.$$
Define
\begin{equation*}\begin{split}
\nu(\delta)=\inf\left\{\mathcal{S}_S(u): u\geq 0, \inf_{h\in {\mathcal M}_S}\|L_{g_{\mathbb{S}^n}}^{\frac{1}{2}}(u-h)\|_2^2\leq \delta\|L_{g_{\mathbb{S}^n}}^{\frac{1}{2}}(u)\|_2^2\right\},
\end{split}\end{equation*}
where $$\mathcal{S}_{S}(u)=\frac{\mathcal{S}_{1,n}\|L_{g_{\mathbb{S}^n}}^{\frac{1}{2}}(u)\|_2^2-\|u\|^2_{\frac{2n}{n-2}}}{d^2(u,{\mathcal M}_S)},\ \ d^2(u,\mathcal{M}_S)=\min\{\|L_{g_{\mathbb{S}^n}}^{\frac{1}{2}}(u-h)\|_2^2:\ h \in \mathcal{M}_S \}$$
and $L_{g_{\mathbb{S}^{n}}}=-\Delta_{g_{\mathbb{S}^{n}}}+\frac{n(n-2)}{4}$ denotes the conformal Laplacian operator in $\mathbb{S}^n$.
In the $n$-dimensional sphere $\mathbb{S}^{n}$, given the initial metric $g_0=u_0^{\frac{4}{n-2}}g_{\mathbb{S}^{n}}$, define the Yamabe flow
\begin{equation}\label{YM1}\begin{cases}
&\frac{\partial u(x,t)}{dt}=\left(r_{g(t)}-R_{g(t)}(x)\right)u(x,t),\\
&u(x,0)=u_0,
\end{cases}\end{equation}
where $R_{g(t)}(x)$ is the scalar curvature at the point $x$ of manifold $(S^{n},u(x,t)^{\frac{4}{n-2}}g_{\mathbb{S}^{n}})$  and $g_{\mathbb{S}^n}$ denotes the standard metric on the sphere $\mathbb{S}^n$, $r_{g(t)}$ is the average of $R_{g(t)}$ in $(S^{n},u^{\frac{4}{n-2}}(x,t)g_{\mathbb{S}^{n}} )$. It follows from \cite{Ye} that $u^{\frac{4}{n-2}}(x,t)$ has the long time behavior and has a unique limit $u_{\infty}$ in $W^{1,2}(\mathbb{S}^n)$
when $t\rightarrow +\infty$ with $R_{g_{\infty}}$ equal to positive constant.
\medskip

The total scalar curvature functional $S(t)$ on $(S^{n},u^{\frac{4}{n-2}}(x,t)g_{\mathbb{S}^{n}})$ is given by

$$S(t)=\int_{\mathbb{S}^n}R_{g(t)}(x)dv_{g(x,t)},\ \ g(x,t)=u(x,t)^{\frac{4}{n-2}}g_{\mathbb{S}^{n}}.$$
For notational simplicity, we write $u(t)$ for $u(x,t)$, $g(t)$ for $g(x,t)$. We claim that $S(t)$ is decreasing with respect to the variable $t$. Recall that the conformal Laplacian in $(S^{n},u(t)^{\frac{4}{n-2}}g_{\mathbb{S}^{n}})$ is given by $L_{g(t)}=-\Delta_{g(t)}+\frac{n-2}{4(n-1)}R_{g(t)}$ with $g(t)=u(t)^{\frac{4}{n-2}}g_{\mathbb{S}^{n}}$, where $\Delta_{g(t)}$ denotes the Laplace-Beltrami operator associated with the metric $g(t)$. It satisfies the conformal law
$$L_{g(t)}(\phi)=u^{-\frac{n+2}{n-2}}(t)L_{g_{\mathbb{S}^{n}}}(u(t)\phi),\ \ \forall\ \phi\in C^{\infty}(\mathbb{S}^n).$$
Let $\phi=1$, we immediately get that $R_{g(t)}=\frac{4(n-1)}{n-2}u^{-\frac{n+2}{n-2}}(t)L_{g_{\mathbb{S}^{n}}}(u(t))$.   Then we can write
\begin{equation}\begin{split}
S(t)&=\int_{\mathbb{S}^{n}}L_{g_{\mathbb{S}^n}}(u(t))u^{-\frac{n+2}{n-2}}(t)dv_{g(t)}\\
&=\int_{\mathbb{S}^{n}}L_{g_{\mathbb{S}^n}}(u(t))u(t)dv_{g_{\mathbb{S}^n}}.
\end{split}\end{equation}
Hence we deduce that $$S'(t)=2\int_{\mathbb{S}^{n}}L_{g_{\mathbb{S}^n}}(u(t))\frac{\partial u(t)}{dt}dv_{g_{\mathbb{S}^n}}.$$ Applying the Yamabe flow equation \eqref{YM1}, we immediately deduce that
\begin{equation}\begin{split}
S'(t)&=2\int_{\mathbb{S}^{n}}\left(r_{g(t)}-R_{g(t)}(x)\right)u(t)L_{g_{\mathbb{S}^n}}(u(t))dv_{g_{\mathbb{S}^n}}\\
&=2\int_{\mathbb{S}^{n}}\left(r_{g(t)}-R_{g(t)}(x)\right)u^{-\frac{n+2}{n-2}}(t) L_{g_{\mathbb{S}^n}}(u(t)) u^{\frac{2n}{n-2}}(t)dv_{g_{\mathbb{S}^n}}\\
&=\frac{2(n-2)}{4(n-1)}\int_{\mathbb{S}^{n}}\left(r_{g(t)}-R_{g(t)}(x)\right)R_{g(t)}dv_{g(t)}\\
&=-\frac{2(n-2)}{4(n-1)}\int_{\mathbb{S}^{n}}\left(r_{g(t)}-R_{g(t)}(x)\right)^2dv_{g(t)}\leq 0,
\end{split}\end{equation}
where the last equality holds since $\int_{\mathbb{S}^{n}}\left(r_{g(t)}-R_{g(t)}(x)\right)dv_{g(t)}=0$ from the definition of $r_{g(t)}$.
This proves
$$S(t)=\int_{\mathbb{S}^n}R_{g(t)}(x)dv_{g(t)}=\int_{\mathbb{S}^{n}}L_{g_{\mathbb{S}^n}}(u(t))u(t)dv_{g_{\mathbb{S}^n}}=\int_{\mathbb{S}^n}
\left(|\nabla_{g_{\mathbb{S}^n}}u(t)|^2+\frac{n(n-2)}{4}|u(t)|^2\right)dv_{g_{\mathbb{S}^n}}$$ is non-increasing. Next, we furthermore claim that the volume $$V(t)=\int_{\mathbb{S}^n}dv_{g(t)}=\int_{\mathbb{S}^{n}}u^{\frac{2n}{n-2}}(t)dv_{g_{\mathbb{S}^n}}$$ remains unchanged. Direct computation yields
\begin{equation}\begin{split}
V'(t)&=\frac{2n}{n-2}\int_{\mathbb{S}^{n}}u^{\frac{n+2}{n-2}}(t)\frac{\partial u(t)}{dt}dv_{g_{\mathbb{S}^n}}\\
&=\frac{2n}{n-2}\int_{\mathbb{S}^{n}}u^{\frac{n+2}{n-2}}(t)\left(r_{g(t)}-R_{g(t)}(x)\right)u(t)dv_{g_{\mathbb{S}^n}}\\
&=\frac{2n}{n-2}\int_{\mathbb{S}^{n}}\left(r_{g(t)}-R_{g(t)}(x)\right)dv_{g(t)}=0.
\end{split}\end{equation}
In summary, the Yamabe flow can make the Sobolev energy $$\int_{\mathbb{S}^n}
\left(|\nabla_{g_{\mathbb{S}^n}}u(t)|^2+\frac{n(n-2)}{4}|u(t)|^2\right)dv_{g_{\mathbb{S}^n}}$$ decrease and make the volume $\int_{\mathbb{S}^n}dv_{g(t)}=\int_{\mathbb{S}^{n}}u^{\frac{2n}{n-2}}(t)dv_{g_{\mathbb{S}^n}}$ unchanged.
\medskip

Recall that the Yamabe flow has the long time behavior and there exists a unique limit $u_{\infty}\in W^{1,2}(\mathbb{S}^n)$ such that
$u(t)$ converges to $u_{\infty}$ in $W^{1,2}(\mathbb{S}^n)$ with $R_{g_{\infty}}$ equal to some positive constant. Hence $u_{\infty}$ satisfies Yamabe equation on the sphere:
$$-\Delta_{g_{\mathbb{S}^n}}u_{\infty}+\frac{n(n-2)}{4}u_{\infty}=R_{g_{\infty}}u^{\frac{n+2}{n-2}}_{\infty}.$$
According to well-known classification result of Yamabe equation on the sphere, we deduce that $u_{\infty}$ is in fact some extremal of Sobolev inequality on $\mathbb{S}^n$.
Obviously
$$d^{2}(u(t),\mathcal{M}_S)=\inf_{h\in \mathcal{M}_S}\int_{\mathbb{S}^n}\left(|\nabla_{g_{\mathbb{S}^n}}(u(t)-h)|^2+\frac{n(n-2)}{4}|u(t)-h|^2\right)dv_{g_{\mathbb{S}^n}}\rightarrow 0$$
as $t\rightarrow +\infty$ since $u(t)$ converges to some extremal of Sobolev inequality of sphere in $W^{1,2}(\mathbb{S}^n)$.
\medskip

Now we are in a position to use the Yamabe flow replacing rearrangement flow to establish the relation between the local stability  and the global stability of the Sobolev inequality. For any nonnegative function $u_0\in W^{1,2}(\mathbb{S}^n)$ satisfying $$\inf_{h\in \mathcal{M}_S}\|L_{g_{\mathbb{S}^n}}^{\frac{1}{2}}(u_0-h)\|_2^2> \delta\|L_{g_{\mathbb{S}^n}}^{\frac{1}{2}}(u_0)\|_2^2,$$
we construct the Yamabe flow:\begin{equation}\begin{cases}
&\frac{\partial u(t)}{dt}=\left(r_{g(t)}-R_{g(t)}(x)\right)u(t),\\
&u(0)=u_0.
\end{cases}\end{equation}
Through the previous analysis of the Yamabe flow, we know there exists a time $t_0$ such that
$$\inf_{h\in \mathcal{M}_S}\|L_{g_{\mathbb{S}^n}}^{\frac{1}{2}}(u(t_0)-h)\|_2^2=\delta\|L_{g_{\mathbb{S}^n}}^{\frac{1}{2}}(u(t_0))\|_2^2$$
since $$\frac{ \mathop {\inf }\limits_{h\in \mathcal{M}_S}\|L_{g_{\mathbb{S}^n}}^{\frac{1}{2}}(u(t)-h)\|_2^2}{\|L_{g_{\mathbb{S}^n}}^{\frac{1}{2}}(u(t))\|_2^2}$$ continuously converges to zero as $t\rightarrow +\infty$. This together with $S(t)$ being non-increasing and $V(t)$ remaining unchanged gives  that
\begin{equation}\begin{split}
&\frac{S_{1,n}\|L_{g_{\mathbb{S}^n}}^{\frac{1}{2}}(u_0)\|_2^2-\|u_0\|^2_{\frac{2n}{n-2}}}{d^2(u_0,M_S)}\\
&\ \ \geq \frac{S_{1,n}\|L_{g_{\mathbb{S}^n}}^{\frac{1}{2}}(u_0)\|_2^2-\|u_0\|^2_{\frac{2n}{n-2}}}{\|L_{g_{\mathbb{S}^n}}^{\frac{1}{2}}(u_0)\|_2^2}\\
&\ \ \geq\frac{S_{1,n}\|L_{g_{\mathbb{S}^n}}^{\frac{1}{2}}(u(t_0))\|_2^2-\|u(t_0)\|^2_{\frac{2n}{n-2}}}{\|L_{g_{\mathbb{S}^n}}^{\frac{1}{2}}(u(t_0))\|_2^2}\\
&\ \ =\delta \frac{S_{1,n}\|L_{g_{\mathbb{S}^n}}^{\frac{1}{2}}(u(t_0))\|_2^2-\|u(t_0)\|^2_{\frac{2n}{n-2}}}{d^2(u(t_0),\mathcal{M}_S)}\geq \delta\nu(\delta).
\end{split}\end{equation}
This establishes the quantitative relation between the local stability and global stability of Sobolev inequality on the sphere $\mathbb{S}^n$. In fact, we derive that for any non-negative function $u\in W^{1,2}(\mathbb{S}^n)$, there holds
$$S_{1,n}\|L_{g_{\mathbb{S}^n}}^{\frac{1}{2}}(u)\|_2^2-\|u\|^2_{\frac{2n}{n-2}}\geq \sup_{\delta \in (0,1)}\{\delta \nu(\delta)\}d^2(u,M_S).$$

\subsection{CR-Yamabe flow}
Let us recall some basic concept for CR manifold and CR Yamabe flow. On can refer to \cite{JeLe1,P, P1} for detailed introduction.
\medskip

Consider the CR sphere $\mathbb{S}^{2n+1}$ in the complex space $\mathbb{C}^{n+1}$ equipped with the contact form $\theta_0$. Then
$(\mathbb{S}^{2n+1}, \theta_0)$ is in fact a compact CR manifold of dimension $2n+1$ with a given contact form $\theta_0$.
CR Yamabe flow first introduced by Jerison and Lee in \cite{JeLe1} is defined by
\begin{equation}\begin{cases}
&\frac{\partial\theta(t)}{\partial t}=\left(r_{\theta(t)}-R_{\theta(t)}\right)\theta(t),\\
&\theta(0)=\theta_0,
\end{cases}\end{equation}
where $R_{\theta(t)}(x)$ is the Webster scalar curvature at the point $x$ of CR manifold $(\mathbb{S}^{2n+1}, \theta(t))$ and
$$r_{\theta(t)}=\frac{\int_{\mathbb{S}^{2n+1}}R_{\theta(t)}(x)dV_{\theta(t)}}{\int_{\mathbb{S}^{2n+1}}dV_{\theta(t)}}$$
is the average of $R_{\theta(t)}(x)$ on $(\mathbb{S}^{2n+1}, \theta(t))$. If we write
$\theta(t)=u(t)^{\frac{2}{n}}\theta_0$ for some function $u(t)$, then $dV_{\theta(t)}=u^{2+\frac{2}{n}}dV_{\theta_0}$ and
the above CR-Yamabe flow can be written as
\begin{equation}\begin{cases}
&\frac{\partial u(t)}{\partial t}=\frac{n}{2}\left(r_{\theta(t)}-R_{\theta(t)}\right)u(t),\\
&u(0)=1.
\end{cases}\end{equation}
Since $\theta(t)=u(t)^{\frac{2}{n}}\theta_0$, we have the CR Yamabe equation:
\begin{equation}\label{CR-equation}
-(2+\frac{2}{n})\Delta_{\theta_0}u(t)+R_{\theta_0}u(t)=R_{\theta(t)}u^{1+\frac{2}{n}}(t),
\end{equation}
where $\Delta_{\theta_0}$ denotes the corresponding sub-Laplacian operator.
The total Webster scalar curvature  functional $S_{\theta(t)}$ on $(S^{2n+1},u(t)^{\frac{2}{n}}\theta_0)$ is given by
$$S_{\theta(t)}=\int_{\mathbb{S}^{2n+1}}R_{\theta(t)}(x)dV_{\theta(t)},\  \theta(t)=u(t)^{\frac{2}{n}}\theta_0.$$
We claim that $S_{\theta(t)}$ is decreasing with respect to the variable $t$. Through the CR Yamabe equation, we can write

\begin{equation}\begin{split}
S_{\theta(t)}&=\int_{\mathbb{S}^{2n+1}}\big(-(2+\frac{2}{n})\Delta_{\theta_0}u(t)+R_{\theta_0}u(t)\big)u^{-1-\frac{2}{n}}(t)dV_{\theta(t)}\\
&=\int_{\mathbb{S}^{2n+1}}\big(-(2+\frac{2}{n})\Delta_{\theta_0}u(t)+R_{\theta_0}u(t)\big)u(t)dV_{\theta_0}.
\end{split}\end{equation}
Hence $$\frac{dS_{\theta(t)}}{dt}=2\int_{\mathbb{S}^{2n+1}}\big(-(2+\frac{2}{n})\Delta_{\theta_0}u(t)+R_{\theta_0}u(t)\big)\frac{\partial u}{\partial t}dV_{\theta_0}.$$ Applying the CR Yamabe flow, we immediately deduce that
\begin{equation}\begin{split}
\frac{dS_{\theta(t)}}{dt}&=n\int_{\mathbb{S}^{2n+1}}\left(r_{\theta(t)}-R_{\theta(t)}\right)u(t)\big(-(2+\frac{2}{n})\Delta_{\theta_0}u(t)+R_{\theta_0}u(t)\big)dV_{\theta_0}\\
&=n\int_{\mathbb{S}^{2n+1}}\left(r_{\theta(t)}-R_{\theta(t)}\right)R_{\theta(t)}u^{2+\frac{2}{n}}(t)dV_{\theta_0}\\
&=n\int_{\mathbb{S}^{2n+1}}\left(r_{\theta(t)}-R_{\theta(t)}\right)R_{\theta(t)}dV_{\theta(t)}\\
&=-n\int_{\mathbb{S}^{2n+1}}\left(r_{\theta(t)}-R_{\theta(t)}\right)^2dV_{\theta(t)}\leq 0,
\end{split}\end{equation}
where the last equality holds since $\int_{\mathbb{S}^{2n+1}}\left(r_{\theta(t)}-R_{\theta(t)}\right)dV_{\theta(t)}=0$ from the definition of $r_{\theta(t)}$.
This proves
\begin{equation}\begin{split}
S_{\theta(t)}&=\int_{\mathbb{S}^{2n+1}}\big(-(2+\frac{2}{n})\Delta_{\theta_0}u(t)+R_{\theta_0}u(t)\big)u^{-1-\frac{2}{n}}(t)u^{2+\frac{2}{n}}dV_{\theta_0}\\
&=\int_{\mathbb{S}^{2n+1}}\big(-(2+\frac{2}{n})\Delta_{\theta_0}u(t)+R_{\theta_0}u(t)\big)u(t)dV_{\theta_0}\\
&=\int_{\mathbb{S}^{2n+1}}\Big((2+\frac{2}{n})|\nabla_{\theta_0}u(t)|^2+R_{\theta_0}|u(t)|^2\Big)dV_{\theta_0}
\end{split}\end{equation}
is non-increasing. Next, we furthermore claim that the volume $$V_{\theta(t)}=\int_{\mathbb{S}^{2n+1}}dV_{\theta(t)}=\int_{\mathbb{S}^{2n+1}}u^{2+\frac{2}{n}}dV_{\theta_0}$$ remains unchanged.
Applying the CR Yamabe flow equation again, we derive
\begin{equation}\begin{split}
\frac{dV_{\theta(t)}}{dt}&=(2+\frac{2}{n})\int_{\mathbb{S}^{2n+1}}u^{1+\frac{2}{n}}(x,t)\frac{\partial u}{\partial t}dV_{\theta_0}\\
&=(2+\frac{2}{n})\int_{\mathbb{S}^{2n+1}}u^{2+\frac{2}{n}}\left(r_{\theta(t)}-R_{\theta(t)}\right)dV_{\theta_0}\\
&=(2+\frac{2}{n})\int_{\mathbb{S}^{2n+1}}\left(r_{\theta(t)}-R_{\theta(t)}\right)udV_{\theta(t)}=0.
\end{split}\end{equation}
In summary, the CR Yamabe flow can make the Sobolev energy $$\int_{\mathbb{S}^{2n+1}}\Big((2+\frac{2}{n})|\nabla_{\theta_0}u(t)|^2+R_{\theta_0}|u(t)|^2\Big)dV_{\theta_0}$$ decrease and keep the volume $$\int_{\mathbb{S}^{2n+1}}u^{2+\frac{2}{n}}dV_{\theta_0}$$  unchanged.
\medskip

If we choose the $\theta_0$ to be the standard contact form of CR sphere $\mathbb{S}^{2n+1}$, that is to say $(\mathbb{S}^{2n+1}, \theta_0)$ is $CR$ sphere $\mathbb{S}^{2n+1}$ and $R_{\theta_0}=n(n+1)$. Then it follows that
$$\big(-(2+\frac{2}{n})\Delta_{\theta_0}+R_{\theta_0}\big)=\frac{n}{4(n+1)}\big(-\frac12 \sum_{j=1}^{n+1} \left(\overline{T_j}T_j + T_j \overline{T_j}\right) + \frac{n^2}4\big)=\frac{n}{4(n+1)} \mathcal L.$$  In \cite{P}, the author proved that the convergence of CR Yamabe flow on the CR sphere $\mathbb{S}^{2n+1}$ (see \cite{P1} for an alternative proof). Hence it follows that there exists a unique limit $u_{\infty}$ such that $u(x,t)$ converges to $u_{\infty}$ in $S^{1}(\mathbb{S}^{2n+1})$ with $R_{\theta(\infty)}$ equaling to some positive constant. Hence $u_{\infty}$ satisfies CR Yamabe equation on CR sphere $\mathbb{S}^{2n+1}$:
$$\frac{n}{4(n+1)} \mathcal L(u_{\infty})=R_{\theta_{\infty}}u^{1+\frac{2}{n}}_{\infty}.$$
According to the well-known classification result of CR Yamabe equation on the CR sphere $\mathbb{S}^{2n+1}$ obtained by Jerison and Lee in \cite{JeLe}, we deduce that $u_{\infty}$ is in fact some extremal of Sobolev inequality on the CR sphere $\mathbb{S}^{2n+1}$.

\subsection{The local stability of Sobolev inequality on the CR sphere can imply the corresponding global stability}

Now we are in a position to use CR-Yamabe flow to establish the relation between the local stability of Sobolev inequality on the CR sphere and the corresponding global stability.  Define
\begin{equation*}\begin{split}
\nu_{*}(\delta)=\inf\left\{\mathcal{S}_S(u): u\geq 0,
\inf_{h\in \mathcal{M}_{*}}\|\mathcal L^{\frac{1}{2}}(u-h)\|_2^2\leq \delta\|\mathcal L^{\frac{1}{2}}(u)\|_2^2\right\},
\end{split}\end{equation*}
where $$\mathcal{S}_{*}(u)=\frac{\mathcal E[u]
-|\mathbb{S}^{2n+1}|^{\frac{Q}{2}}(\frac{Q-2}{2})^2  \left( \int_{\Sph^{2n+1}} |u|^{\frac{2Q}{Q-2}} \,d\xi \right)^{\frac{Q-2}{Q}}}
{  \mathop {\inf }\limits_{g\in \mathcal{M}_{*}} \mathcal E[u-g].}$$ and $\mathcal{M}_{*}$ denotes the set of real-valued extremal functions of Sobolev inequality on the CR sphere $\mathbb{S}^{2n+1}$.
For any nonnegative $u_0\in S^1(\mathbb{S}^{2n+1})$ satisfying $$\inf_{h\in \mathcal{M}_{*}}\|\mathcal L^{\frac{1}{2}}(u_0-h)\|_2^2> \delta\|\mathcal L^{\frac{1}{2}}(u_0)\|_2^2,$$
we consider the initial CR manifold $(\mathbb{S}^{2n+1}, u_0^{\frac{2}{n}}\theta_0)$, where $\theta_0$ is the standard contact form of the CR sphere $\mathbb{S}^{2n+1}$. We construct the CR Yamabe flow as follows :\begin{equation}\begin{cases}
&\frac{\partial u}{dt}=\frac{n}{2}\left(r_{\theta(t)}-R_{\theta(t)}(x)\right)u,\\
&u(0)=u_0.
\end{cases}\end{equation}
Through earlier analysis of the CR Yamabe flow, we know $$\frac{
  \mathop {\inf }\limits_{h\in \mathcal{M}_{*}}\|\mathcal L^{\frac{1}{2}}(u(t)-h)\|_2^2}{\|\mathcal L^{\frac{1}{2}}(u)\|_2^2}$$ continuously converges to zero as $t\rightarrow +\infty$ since
$$\|L^{\frac{1}{2}}(u(t))\|_2^2\geq |\mathbb{S}^{2n+1}|^{\frac{Q}{2}}(\frac{Q-2}{4})^2\|u(t)\|_{\frac{2Q}{Q-2}}^{\frac{Q-2}{Q}}=(\frac{Q-2}{4})^2\|u_0\|_{\frac{2Q}{Q-2}}^{\frac{Q-2}{Q}}.$$
Then it follows that there exists a time $t_0$ such that
$$\inf_{h\in \mathcal{M}_{*}}\|\mathcal L^{\frac{1}{2}}(u(t_0)-h)\|_2^2=\delta\|\mathcal L^{\frac{1}{2}}(u(t_0))\|_2^2$$
since $$\frac{  \mathop {\inf }\limits_{h\in \mathcal{M}_{*}^{\rm rea}}\|\mathcal L^{\frac{1}{2}}(u(t)-h)\|_2^2}{\|\mathcal L^{\frac{1}{2}}(u(t))\|_2^2}$$  converges continuously  to zero as $t\rightarrow +\infty$. This together with $S_{\theta(t)}$ being non-increasing and $V(t)$ remaining unchanged gives that
\begin{equation}\begin{split}
&\frac{\|\mathcal L^{\frac{1}{2}}(u_0)\|_2^2-|\mathbb{S}^{2n+1}|^{\frac{Q}{2}}(\frac{Q-2}{4})^2\|u_0\|^2_{\frac{2Q}{Q-2}}}{\mathop {\inf }\limits_{g\in \mathcal{M}_\ast} \mathcal E[u_0-g]}\\
&\ \ \geq \frac{\|\mathcal L^{\frac{1}{2}}(u_0)\|_2^2-|\mathbb{S}^{2n+1}|^{\frac{Q}{2}}(\frac{Q-2}{4})^2\|u_0\|^2_{\frac{2Q}{Q-2}}}{\|\mathcal L^{\frac{1}{2}}(u_0)\|_2^2}\\
&\ \ \geq\frac{\|\mathcal L^{\frac{1}{2}}(u(t_0))\|_2^2-|\mathbb{S}^{2n+1}|^{\frac{Q}{2}}(\frac{Q-2}{4})^2\|u(t_0)\|^2_{\frac{2Q}{Q-2}}}{\|\mathcal L^{\frac{1}{2}}(u(t_0))\|_2^2}\\
&\ \ =\delta \frac{\|\mathcal L^{\frac{1}{2}}(u(t_0))\|_2^2-|\mathbb{S}^{2n+1}|^{\frac{Q}{2}}(\frac{Q-2}{4})^2\|u(t_0)\|^2_{\frac{2Q}{Q-2}}}{  \mathop {\inf }\limits_{g\in \mathcal{M}_\ast} \mathcal E[u(t_0)-g]}\geq \delta\nu(\delta).
\end{split}\end{equation}
This establishes the quantitative relation between the local stability and global stability of Sobolev inequality on the CR sphere $\mathbb{S}^{2n+1}$. In conclusion, we derive that for any non-negative $u_0\in S^{1}(\mathbb{S}^{2n+1})$, there holds
$$\|\mathcal L^{\frac{1}{2}}(u_0)\|_2^2-|\mathbb{S}^{2n+1}|^{\frac{Q}{2}}(\frac{Q-2}{4})^2\|u_0\|^2_{\frac{2Q}{Q-2}}\geq \sup_{\delta \in (0,1)}\{\delta \nu(\delta)\}\inf_{g\in \mathcal{M}_\ast} \mathcal E[u_0-g].$$

\subsection{From nonnegative functions to arbitrary functions}\label{sec:posneg}
In this subsection we will prove the stability of Sobolev inequalities on the CR sphere with
optimal dimension-dependent constants for general real-valued functions, namely we shall give the
proof of Theorem \ref{sob stability}. Recall
$$\mathcal{S}_{\ast}(u)=\frac{\mathcal{E}[u]-|\mathbb{S}^{2n+1}|^{\frac{Q}{2}}(\frac{Q-2}{4})^2\|u\|_{{L^{2^*}(\mathbb{S}^{2n+1})}}^2}{   \mathop {\inf }\limits_{g\in \mathcal{M}_{\ast}}\mathcal{E}[u-g]},\ \ 2^{*}=\frac{2Q}{Q-2}$$
and denote by $C_{BE}$ the optimal constant for stability of the Sobolev inequality, i.e.,
$$C_{BE}=\inf_{u\in S^{1}(\mathbb{S}^{2n+1})}\mathcal{S}_{\ast}(u).$$
Similarly, we denote by $C_{\rm BE}^{\rm pos}$ the optimal constant  when restricted to nonnegative functions $u$. We will establish the relationship between these two optimal constants.

\begin{proposition}\label{Prop:BE} For any $n\geq 1$,
\[\label{eq:cBE1}
C_{\rm BE}\ge \min\left\{ \textstyle \frac12  C_{\rm BE}^{\rm pos}, 1-2^{-\frac 2Q} \right\} .
\]
\end{proposition}
\begin{proof}
For $u\in S^{1}(\mathbb{S}^{2n+1})$,
\[
\mathcal{S}_{\ast}(u) \mathop {\inf }\limits_{g\in \mathcal{M}_{\ast}^{rea}}\mathcal{E}[u-g]=\mathcal{E}[u]-|\mathbb{S}^{2n+1}|^{\frac{Q}{2}}(\frac{Q-2}{4})^2\|u\|_{L^{2^*}( \mathbb{S}^{2n+1})}^{2}.
\]
By homogeneity, we can assume that $\|u\|_{L^{2^*}(\mathbb{H}^n)}=1$.
Let $u_\pm$ denote the positive and negative parts of $u$,
there naturally hold that
$$\mathcal{E}[u]
=\mathcal{E}[u_+]+\mathcal{E}[u_-] \ \ {\rm and} \ \
\|u\|_{L^{2^*}(\mathbb{S}^{2n+1})}^{2^*}=\| u_+\|_{L^{2^*}(\mathbb{S}^{2n+1})}^{2^*}+\|u_-\|_{L^{2^*}(\mathbb{S}^{2n+1})}^{2^*}.$$
Thus, we deduce that
\[
\mathcal{S}_{\ast}(u_+)\mathop {\inf }\limits_{g\in \mathcal{M}_{\ast}}\mathcal{E}[u_+-g]=\mathcal{E}[u_+]
-|\mathbb{S}^{2n+1}|^{\frac{Q}{2}}(\frac{Q-2}{4})^2
\|u_+\|_{L^{2^*}(\mathbb{S}^{2n+1})}^{2^*}
\]
and
\[
\mathcal{S}_{\ast}(u_-) \mathop {\inf }\limits_{g\in \mathcal{M}_{\ast}}\mathcal{E}[u_--g]=\mathcal{E}[u_-]-|\mathbb{S}^{2n+1}|^{\frac{Q}{2}}(\frac{Q-2}{4})^2\|u_-\|_{L^{2^*}(\mathbb{S}^{2n+1})}^{2^*}.
\]
To simplify the notation,  set
\[
m :=\|u_-\|_{L^{2^*}(\mathbb{S}^{2n+1})}^{2^*},
\]
and assume (without loss of generality) that
\begin{equation}\label{m:onehalf}
m\in[0,1/2]\,.
\end{equation}
It yields that
\begin{align*}
\mathcal{S}_{\ast}(u) \mathop {\inf }\limits_{g\in \mathcal{M}_{\ast}}\mathcal{E}[u-g]
&\geq
\mathcal{S}_{\ast}(u_+) \mathop {\inf }\limits_{g\in \mathcal{M}_{\ast}}\mathcal{E}[u_+-g]+
\mathcal{S}_{\ast}(u_-) \mathop {\inf }\limits_{g\in \mathcal{M}_{\ast}}\mathcal{E}[u_--g]\\
&\ \ +|\mathbb{S}^{2n+1}|^{\frac{Q}{2}}(\frac{Q-2}{4})^2((1-m)^\frac{2}{2^*}+m^{\frac{2}{2^*}}-1).
\end{align*}
Combining  the computation of  Proposition 3.10 in \cite{DEFFL} and the expression of $\mathcal{S}_{\ast}(u_-)$,
these give that
\begin{align*}
&\mathcal{S}_{\ast}(u) \mathop {\inf }\limits_{g\in \mathcal{M}_{\ast}}\mathcal{E}[u-g]\geq
\mathcal{S}_{\ast}(u_+) \mathop {\inf }\limits_{g\in \mathcal{M}_{\ast}}\mathcal{E}[u_+-g]+(2-2^{1-\frac{2}{Q}})\mathcal{E}[u_-].
\end{align*}
 Hence, if $g_+\in\mathcal M_\ast$ is optimal for $u_+$, we have
\begin{align*}
\mathcal{S}_{\ast}(f)\mathop {\inf }\limits_{g\in \mathcal{M}_{\ast}}\mathcal{E}[u-g]&\geq
C_{\rm BE}^{\rm pos} \mathcal{E}[u_+-g_+]+(2-2^{1-\frac{2}{Q}})\mathcal{E}[u_-]\\
&\geq\min \{C_{\rm BE}^{\rm pos},\ (2-2^{1-\frac{2}{Q}})\}\left(
\mathcal{E}[u_+-g_+]+\mathcal{E}[u_-]\right)\\
&\geq \frac 12 \min \{C_{\rm BE}^{\rm pos},\ (2-2^{1-\frac{2}{Q}})\}\mathcal{E}[u_+-g_+].
\end{align*}
As a consequence,
\begin{align*}
C_{\rm BE} &\geq \frac 12 \min \{C_{\rm BE}^{\rm pos},\ (2-2^{1-\frac{2}{Q}})\}= \min \{\frac 12 C_{\rm BE}^{\rm pos},\ (1-2^{-\frac{2}{Q}})\}.
\end{align*}

\end{proof}
This establishes the quantitative relation between the local stability of Sobolev inequality in CR sphere $\mathbb{S}^{2n+1}$ for non-negative functions and the global stability of Sobolev inequality in CR sphere $\mathbb{S}^{2n+1}$ for general real-valued functions. In Section 3, we will establish the local stability of Sobolev inequality on the CR sphere $\mathbb{S}^{2n+1}$ with the optimal dimension-dependent constants. This together with Proposition \ref{Prop:BE} yields the stability of Sobolev inequalities on the Heisenberg group with the optimal dimension-dependent constants.

\section{Local stability for nonnegative functions}\label{sec:2}

Our goal in this section is to prove Lemma \ref{local sta}, i.e., a quantitative stability inequality for nonnegative functions close to the manifold of optimizers. We adopted the method  similar to that introduced in  \cite{DEFFL}, where  the authors there  obtain the optimal local stability on the sphere while we prove the
analogues on the CR sphere $\mathbb{S}^{2n+1}$. For the sake of completeness, we state the details. In order to simplify the notation, we write
$$
q = 2^* = \frac{2Q}{Q-2},
\qquad
\theta=q-2=\frac{4}{Q-2}.
$$
By conformal and scaling invariance of the stability of Sobolev inequality on the CR sphere $\mathbb{S}^{2n+1}$, we may assume $f=1+r\geq 0$, and $r$ is
orthogonal with spherical harmonics with degree $0$ and $1$.  In this section, we denote by $d\xi$ the uniform probability measure on $\mathbb{S}^{2n+1}$ and it is easy to check that the following lemma implies Lemma \ref{local sta}.

\begin{lemma}\label{unifboundclose}
There are explicit constants $\epsilon_0>0$ and $\tilde\delta\in(0,1)$ such that for all $n\geq 2$ and for all $-1 \leq r\in \mathrm S^1(\mathbb{S}^{2n+1})$ satisfying
\begin{equation}
\label{eq:smallsphere}
\left( \int_{\mathbb{S}^{2n+1}}{|r|^q}d\xi \right)^{2/q} \leq\tilde\delta
\end{equation}
and
\begin{equation}
\label{eq:orthosphere}
\int_{\mathbb{S}^{2n+1}}rd\xi= \int_{\mathbb{S}^{2n+1}}{\zeta_j\,r}d\xi=\int_{\mathbb{S}^{2n+1}}{\overline{\zeta_j}\,rd\xi}=0 \
\quad j=1,\ldots,n+1\,,
\end{equation}
one has
$$
\mathcal E [1+r]   -(\frac{Q-2}{4})^2\left( \int_{\mathbb{S}^{2n+1}}{|1+r|^q}d\xi\right)^{\frac{2}{q}} \geq \theta \epsilon_0  \mathcal E[r].
$$
\end{lemma}

Now let us prove lemma~\ref{unifboundclose} and we  consider $n\geq 2$ which implies $q\in(2,3]$.
Thus, in this section we will assume that $r$ satisfies the assumptions of Theorem~\ref{unifboundclose}. For any $r\geq -1$, define $r_1$, $r_2$ and $r_3$ by
\be{r1r2r3}
r_1:=\min\{r,\gamma\}\,,\quad r_2:=\min\{(r-\gamma)_+,M-\gamma\}\quad\mbox{and}\quad r_3:=(r-M)_+
\ee
where $\gamma$ and $M$ are parameters such that $0<\gamma<M$.
We need the following lemma
\begin{lemma}[\cite{DEFFL}]\label{Cor:ptw}
Given $\epsilon>0$, $M>0$, and $\gamma\in(0,M/2)$, there is a constant $C_{\gamma,\epsilon,M}>0$ with the following property: if $2\leq q\leq 3$, $r\in [-1,\infty)$, then
\begin{multline}\label{Claim:CuttingEstim}
(1+r)^q - 1 - q\,r \leq \(\tfrac12\,q\,(q-1) + 2\,\gamma\,\theta\) r_1^2 +\(\tfrac12\,q\,(q-1) + C_{\gamma,\epsilon,M}\,\theta\)r_2^2\\
+ 2\,r_1\,r_2 + 2\,(r_1+r_2)\,r_3 + (1+\epsilon\,\theta)\,r_3^q.
\end{multline}
\end{lemma}


We split the deficit
$$\mathcal E [1+r]   -(\frac{Q-2}{4})^2\left( \int_{\mathbb{S}^{2n+1}}{|1+r|^q}d\xi\right)^{\frac{2}{q}}$$ into three parts.

Given two parameters $\epsilon_1$, $\epsilon_2>0$, we apply Lemma \ref{Cor:ptw} with
\be{epsilon1-epsilon2-C}
\gamma=\frac{\epsilon_1}2\,,\quad\epsilon=\epsilon_2\quad\mbox{and}\quad C_{\gamma,\epsilon,M}=C_{\epsilon_1,\epsilon_2}\,.
\ee
In terms of these parameters, we decompose $r=r_1+r_2+r_3$.
By the orthogonality conditions \eqref{eq:orthosphere}, we get
$$
\mathcal E_0[r] = \mathcal E_0[r_1] +\mathcal E_0[r_2] + \mathcal E_0[r_3].
$$
Since $r$ has mean zero, it follows from that
$$
(1+r,1+r) = 1 +(r,r),
$$
 where the inner product symbol $(\cdot,\cdot)$ denotes $\|\cdot \|_{L^2({\mathbb{S}^{2n+1}})}^2$.
Moreover,
$$
(r,r)= (r_1,r_1) +(r_2,r_2) + (r_3,r_3)+ 2[\left(r_1, r_2\right)+\left(r_2, r_3\right)+\left(r_1, r_3\right)].
$$
According to  Lemma \ref{Cor:ptw} and using again the fact that $r$ has mean zero, we have
\begin{align*}
\int_{\mathbb{S}^{2n+1}}{|1+r|^q} d\xi&
 \leq 1 +\(\tfrac12\,q\,(q-1) + \epsilon_1\,\theta\)  (r_1,r_1) +\(\tfrac12\,q\,(q-1) + C_{\epsilon_1,\epsilon_2}\,\theta\) (r_2,r_2)\\
& \quad + 2[\left(r_1, r_2\right)+\left(r_2, r_3\right)+\left(r_1, r_3\right)]+ (1+\epsilon_2\,\theta) \int_{\mathbb{S}^{2n+1}}{|r_3|^q}d\xi.
\end{align*}
Using $(1+x)^{2/q}\leq 1 + \tfrac2q\,x$ and $q\geq 2$, we obtain
\begin{align*}
\left( \int_{\mathbb{S}^{2n+1}}{|1+r|^q} d\xi \right)^{2/q}
& \leq 1 + (q-1 + \epsilon_1\,\theta)(r_1, r_1)+ (q-1 + C_{\epsilon_1,\epsilon_2}\,\theta) (r_2, r_2)\\
& \quad + 2 [\left(r_1, r_2\right)+\left(r_2, r_3\right)+\left(r_1, r_3\right)]+ \tfrac 2q\,(1+\epsilon_2\,\theta)\int_{\mathbb{S}^{2n+1}}{|r_3|^q} d\xi.
\end{align*}
For any $0<\epsilon_0\leq\theta^{-1}$,
\begin{align*}
\mathcal E[1+r]-(\frac{Q-2}{4})^2\left(\int_{\mathbb{S}^{2n+1}}{|1+r|^q} d\xi\right)^{2/q}
& \geq \theta\,\epsilon_0 \mathcal E[r]
  + (1-\theta\,\epsilon_0) \mathcal E[r_1]-(\frac{Q-2}{4})^2(q-1 + \epsilon_1\,\theta)(r_1,r_1) \\
& \quad + (1-\theta\,\epsilon_0)  \mathcal E[r_2] -(\frac{Q-2}{4})^2(q-1 + C_{\epsilon_1,\epsilon_2}\,\theta) (r_2,r_2)\\
& \quad + (1-\theta\,\epsilon_0)  \mathcal E[r_3] - \tfrac 2q(\frac{Q-2}{4})^2(1+\epsilon_2\,\theta) \int_{\mathbb{S}^{2n+1}}{|r_3|^q} d\xi.
\end{align*}
With another parameter $\sigma_0>0$ ($\sigma_0$ is to be determined) we define
\begin{align*}
I_1 & := (1-\theta\,\epsilon_0)\mathcal E[r_1]  - (\frac{Q-2}{4})^2(q-1 + \epsilon_1\,\theta) (r_1,r_1) +(\frac{Q-2}{4})^2\sigma_0\,\theta \left[(r_2, r_2)+(r_3,r_3)\right]\,,
\\
I_2 & := (1-\theta\,\epsilon_0) \mathcal E[r_2]- (\frac{Q-2}{4})^2\big(q-1 + (\sigma_0 + C_{\epsilon_1,\epsilon_2})\,\theta\big) (r_2,r_2)\,,
\\
I_3 & := (1-\theta\,\epsilon_0) \mathcal E[r_3] - \tfrac 2q (\frac{Q-2}{4})^2 (1+\epsilon_2\,\theta) \int_{\mathbb{S}^{2n+1}}{|r_3|^q}d\xi -(\frac{Q-2}{4})^2\sigma_0\,\theta (r_3,r_3)\,.
\end{align*}
By direct computations, we conclude
\begin{align*}
I_1 & = (1-\theta\,\epsilon_0) \mathcal E_0[r_1]-\frac{Q-2}{4}(1 + \epsilon_0 + \epsilon_1 )  (r_1,r_1) + \frac{Q-2}{4}\sigma_0 \left[(r_2, r_2)+(r_3,r_3)\right]\,,
\\
I_2 & = (1-\theta\,\epsilon_0) \mathcal E_0[r_2] - \frac{Q-2}{4}(1 + \epsilon_0 + \sigma_0 + C_{\epsilon_1,\epsilon_2})(r_2,r_2)\,.
\end{align*}
To summarize, we have
\begin{align*}
\mathcal E[1+r ]  -   (\frac{Q-2}{4})^2\left(\int_{\mathbb{S}^{2n+1}}{|1+r|^q} d\xi \right)^{2/q} \geq \theta\,\epsilon_0  \mathcal E [r ]   + \sum_{k=1}^3 I_k\,.
\end{align*}
In the following we will show that $I_1$, $I_3$ and $I_2$ are nonnegative.

\subsection{Bound on \texorpdfstring{$I_1$}{I1}}\label{sec:i1}
In this subsection, we will prove $I_1\geq 0$.

\begin{proposition}\label{Prop:estI1}

For any $0<\epsilon_0<\tfrac13$, there is a constant $\overline\sigma_0(\gamma,\epsilon_0,\tilde\delta)>0$ depending explicitly on~$\gamma$, $\epsilon_0$ and $\tilde\delta$ such that for all $n\geq 2$ and all $r\in \mathrm S^1(\mathbb{S}^{2n+1})$ such that $r\ge-1$ and satisfying~\eqref{eq:smallsphere} and~\eqref{eq:orthosphere} as in Lemma~\ref{unifboundclose},
\be{epsilon1}
\epsilon_1=\tfrac12\,(1-3\,\epsilon_0)
\ee
and $\sigma_0\ge\overline\sigma_0(\gamma,\epsilon_0,\tilde\delta)$, one has
$$
I_1\geq 0\,.
$$
\end{proposition}
Notice that $\theta=q-2\le1$ with $q=2Q/(Q-2)$ means $n\geq2$. An expression of $\overline\sigma_0$ is given below in~\eqref{sigma0}.
\begin{proof}
 Let $\tilde r_1$ be the orthogonal projection of $r_1$ onto the space of spherical harmonics of degree $k+j\geq 2$, that is,
$$
\tilde r_1 = r_1 - \int_{\mathbb{S}^{2n+1}}{r_1}d\xi-\frac{Q}{2}\sum_{j=1}^{n+1}\left(r_1, \zeta_j\right)\zeta_j-\frac{Q}{2}\sum_{j=1}^{n+1}\left(r_1, \overline{\zeta_j}\right)\overline{\zeta_j},
$$
where $\{\sqrt{Q/2}\zeta_j\}_{j=1}^{n+1}$ is a family of $L^2$ normalized orthogonal basis of space $\mathcal{H}_{1,0}$ and
$\{\sqrt{Q/2}\bar\zeta_j\}_{j=1}^{n+1}$ is a  family of $L^2$ normalized orthogonal basis of space $\mathcal{H}_{0,1}$.
Then
\begin{align*}
I_1 &
= (1-\theta\,\epsilon_0) \mathcal E_0[\widetilde{r}_1]- \frac{Q-2}{4}(1 + \epsilon_0 + \epsilon_1 ) (\widetilde{r}_1,\widetilde{ r}_1)+\frac{Q-2}{4}\sigma_0 \left[(r_2, r_2)+(r_3,r_3)\right] \\
&\ \ \  - \frac{Q-2}{4}(1 + \epsilon_0 + \epsilon_1 ) \left( \int_{\mathbb{S}^{2n+1}}{r_1}d\xi \right)^2-\frac{(Q-2)}{4}\frac{Q}{2}\big(
(1+\theta)\,\epsilon_0 + \epsilon_1 \big) \sum_{j=1}^{n+1}\left(|\left(r_1, \zeta_j\right)|^2 + |\left(r_1, \overline{\zeta_j}\right)|^2\right).
\end{align*}
Since $\widetilde{r}_1\in \bigoplus_{j+k\geq 2} \mathcal H_{j,k}$, we can write $\widetilde{r}_1=\sum\limits _{j+k\geq 2}^{\infty} Y_{j,k}$ and $Y_{j,k}\in H_{j,k}$. Then it holds
\begin{equation}\begin{split}
\mathcal E_0[\widetilde{r}_1]&=\frac12 \int_{\Sph^{2n+1}} \left( \sum_{j=1}^{n+1} \left( |T_j \widetilde{r}_1|^2 + |\overline{T_j} \widetilde{r}_1|^2 \right)\right)d\xi\\
&=\left(\mathcal{L}'(\widetilde{r}_1), \widetilde{r}_1\right)=\left(\mathcal{L}(\widetilde{r}_1), \widetilde{r}_1\right)-\frac{n^2}{4}(\widetilde{r}_1,\widetilde{ r}_1)\\
&=\left(\mathcal{L}\left(\sum\limits _{j+k\geq 2}^{\infty} Y_{j,k}\right), \sum\limits _{l+m\geq 2}^{\infty} Y_{l,m}\right)-\frac{n^2}{4}(\widetilde{r}_1, \widetilde{r}_1)\\
&=\sum\limits _{j+k\geq 2}^{\infty}\lambda_{j,k}\left(Y_{j,k},Y_{j,k}\right)-\frac{n^2}{4}(\widetilde{r}_1, \widetilde{r}_1)\\
&\geq \lambda_{2,0}\sum\limits _{j+k\geq 2}^{\infty}\left(Y_{j,k},Y_{j,k}\right)-\frac{n^2}{4}(\widetilde{r}_1, \widetilde{r}_1)\\
&=\frac{Q-2}{2}(\widetilde{r}_1,\widetilde{ r}_1).
\end{split}\end{equation}
Combining the above estimates, we derive that
\begin{equation}\begin{split}
& I_1\geq \big(\frac{Q-2}{2}(1-\theta\,\epsilon_0) - \frac{Q-2}{4}(1 + \epsilon_0 + \epsilon_1 ) \big) (\widetilde{r}_1,\widetilde{ r}_1) + \frac{Q-2}{4}\sigma_0 \left[(r_2, r_2)+(r_3,r_3)\right]\\
& \ \ \  - \frac{Q-2}{4}(1 + \epsilon_0 + \epsilon_1 ) \left( \int_{\mathbb{S}^{2n+1}}{r_1} d\xi \right)^2
- \frac{(Q-2)}{4}\frac{Q}{2}\big(
(1+\theta)\,\epsilon_0 + \epsilon_1 \big) \sum_{j=1}^{n+1}\left(|\left(r_1, \zeta_j\right)|^2 + |\left(r_1, \overline{\zeta_j}\right)|^2\right).
\end{split}\end{equation}

By the definition \eqref{epsilon1} of   $\epsilon_1$ and for any $\epsilon_0<\tfrac13$, using $\theta\le1$ ($n\geq 2$), we get
\begin{equation}
\label{eq:i1boundmain}
\frac{Q-2}{2}(1-\theta\,\epsilon_0) - \frac{Q-2}{4}(1 + \epsilon_0 + \epsilon_1 ) \ge \frac {Q-2}{8}(1 - 3 \epsilon_0)  = \frac {Q-2}{4}\epsilon_1>0.
\end{equation}
Combining
$$
(\widetilde{r}_1,\widetilde{ r}_1)=(r_1,r_1)-\left( \int_{\mathbb{S}^{2n+1}}{r_1} d\xi \right)^2 -\frac{Q}{2} \sum_{j=1}^{n+1}\left(|\left(r_1, \zeta_j\right)|^2 + |\left(r_1, \overline{\zeta_j}\right)|^2\right),
$$
it is easy to obtain that
\begin{align*}
 I_1\geq& \frac{Q-2}{4}\epsilon_1(\widetilde{r}_1,\widetilde{ r}_1)+\frac{Q-2}{4}\sigma_0 \left[(r_2, r_2)+(r_3,r_3)\right]-\frac{Q-2}{4}(1 + \epsilon_0 + \epsilon_1 ) \left( \int_{\mathbb{S}^{2n+1}}{r_1}  d\xi \right)^2\\
&\hspace*{3cm}
-\frac{(Q-2)}{4}\frac{Q}{2}\big(
(1+\theta)\,\epsilon_0 + \epsilon_1 \big) \sum_{j=1}^{n+1}\left(|\left(r_1, \zeta_j\right)|^2 + |\left(r_1, \overline{\zeta_j}\right)|^2\right)\\
\geq& \frac{Q-2}{4}\epsilon_1(r_1, r_1)+\frac{Q-2}{4}\sigma_0 \left[(r_2, r_2)+(r_3,r_3)\right]-\frac{Q-2}{4}(1 + \epsilon_0+2\epsilon_1) \left(\int_{\mathbb{S}^{2n+1}}{r_1} d\xi \right)^2\\
&\hspace*{3cm}
-\frac{(Q-2)Q}{4}(\epsilon_0+\epsilon_1)\sum_{j=1}^{n+1}\left(|\left(r_1, \zeta_j\right)|^2 + |\left(r_1, \overline{\zeta_j}\right)|^2\right).
\end{align*}

Let~$Y$ be one of the functions $1$, $a\cdot  \zeta$ and $b \cdot \bar{\zeta}$ with $a,b\in\C^{n+1}$.
Since $r$ is orthogonal to $\mathcal{H}_{j,k}$ for $j+k\leq 1$, hence $\left(r,Y\right)=0$.
Then it follows that
$$ |\left(r_1, Y\right)|^2=|\left(r_2+r_3, Y\right)|^2\leq\| Y\|_{L^4({\mathbb{S}^{2n+1}})}^2  \mu\big(\{r_2+r_3>0\}\big)^{1/2}(r_2+r_3,r_2+r_3).
$$
Since $\{r_2+r_3>0\}\subset\{r_1\geq\gamma\}$, there holds
$$
\mu(\{r_2+r_3>0\}) \leq \mu(\{r_1\geq \gamma\}) \leq \frac1{\gamma^2} (r_1, r_1).
$$
Thus we have
\begin{equation}\label{Y}
|\left(r_1, Y\right)|^2 \leq \| Y\|_{L^4({\mathbb{S}^{2n+1}})}^2 \frac{\sqrt{2 \tilde{\delta}}}\gamma (r_1,r_1)^{\frac{1}{2}}\((r_2, r_2)+(r_3,r_3)\)^{\frac{1}{2}}
\end{equation}
If $Y=1$, then clearly $\| Y\|_{L^4({\mathbb{S}^{2n+1}})} =1$ and~\eqref{Y} gives
$$
\left(\int_{\mathbb{S}^{2n+1}}{r_1} d\xi \right)^2 \leq \frac{\sqrt{2\,\tilde\delta}}\gamma (r_1,r_1)^{\frac{1}{2}}\left[(r_2, r_2)+(r_3,r_3)\right]^{\frac{1}{2}}\,.
$$
If $Y=a\cdot\zeta$, we write $a=\left(a_1+ib_1, a_2+i b_2,\cdot\cdot\cdot, a_{n+1}+ib_{n+1}\right)$ and $\zeta=\left(x_1+iy_1, x_2+iy_2,\cdot\cdot\cdot,x_{n+1}+iy_{n+1}\right)$.
Then $$\left(a\cdot\zeta\right)=\left(a_1x_1-b_1y_1+i(a_1y_1+b_1x_1), \cdot\cdot\cdot, a_{n+1}x_{n+1}-b_{n+1}y_{n+1}+i(a_{n+1}y_{n+1}+b_{n+1}x_{n+1})\right)$$
and
\begin{equation}\begin{split}
|\left(a\cdot\zeta\right)|^2&=\sum_{j=1}^{n+1}\left(a_j x_j-b_j y_j\right)^2+\sum_{j=1}^{n+1}\left(a_{j}y_{j}+b_{j}x_{j}\right)^2\\
&=\sum_{j=1}^{n+1}\left((a_jx_j)^2+(b_jy_j)^2+(a_jy_j)^2+(b_jx_j)^2\right)\\
&=|\left(a_1,\cdot\cdot\cdot, a_{n+1},b_1,\cdot\cdot\cdot, b_{n+1}\right)\cdot \left(x_1,\cdot\cdot\cdot, x_{n+1}, y_1,\cdot\cdot\cdot, y_{n+1}\right)|^2\\
&\ \ +|\left(a_1,\cdot\cdot\cdot, a_{n+1},b_1,\cdot\cdot\cdot, b_{n+1}\right)\cdot \left(y_1,\cdot\cdot\cdot, y_{n+1}, x_1,\cdot\cdot\cdot, x_{n+1}\right)|^2\\
\end{split}\end{equation}
From \cite{DEFFL}, we know
$$\int_{\mathbb{S}^{2n+1}}|\left(a_1,\cdot\cdot\cdot, a_{n+1},b_1,\cdot\cdot\cdot, b_{n+1}\right)\cdot \left(x_1,\cdot\cdot\cdot, x_{n+1}, y_1,\cdot\cdot\cdot, y_{n+1}\right)|^4 d\xi \leq \frac{3|a|^4}{(2n+2)^2}$$
and $$\int_{\mathbb{S}^{2n+1}}|\left(a_1,\cdot\cdot\cdot, a_{n+1},b_1,\cdot\cdot\cdot, b_{n+1}\right)\cdot \left(y_1,\cdot\cdot\cdot, y_{n+1}, x_1,\cdot\cdot\cdot, x_{n+1}\right)|^4 d\xi\leq \frac{3|a|^4}{(2n+2)^2}.$$
Hence it follows that
$$\left(a\cdot\zeta,a\cdot\zeta\right)^2\leq 2\left(\frac{3|a|^4}{(2n+2)^2}+\frac{3|a|^4}{(2n+2)^2}\right)=\frac{12|a|^4}{Q^2}.$$
Similarly, we also have
$$\int_{\mathbb{S}^{2n+1}}|a\cdot\overline{\zeta}|^4 d\xi
\leq \frac{12|a|^4}{Q^2}.$$

Combining ~\eqref{Y} and $a=\left((r_1,\zeta_1),\cdot\cdot\cdot, (r_1, \zeta_{n+1})\right)$, we obtain
\begin{equation}\begin{split}
\frac{Q}{2}|a|^2=\frac{Q}{2|a|^2}\left(r_1, a\cdot\zeta\right)^2
&\leq \frac{Q}{2|a|^2} \|a\cdot \zeta\|_{L^4(\mathbb{S}^{2n+1})}^2\frac{\sqrt{2 \tilde{\delta}}}\gamma(r_1,r_1)^{\frac{1}{2}}\left((r_2, r_2)+(r_3,r_3)\right)^{\frac{1}{2}}\\
&=\sqrt{3}\frac{\sqrt{2\,\tilde\delta}}\gamma (r_1,r_1)^{\frac{1}{2}}\left[(r_2, r_2)+(r_3,r_3)\right]^{\frac{1}{2}} .
\end{split}\end{equation}
Similarly, if we let $$b=\left((r_1,\overline{\zeta_1}),\cdot\cdot\cdot, (r_1, \overline{\zeta_{n+1}})\right),$$
then $$\frac{Q}{2}|b|^2=\frac{Q}{2|b|^2}\left(r_1, b\cdot\overline{\zeta}\right)^2\le \sqrt{3}\frac{\sqrt{2\,\tilde\delta}}\gamma (r_1,r_1)^{\frac{1}{2}}\left[(r_2, r_2)+(r_3,r_3)\right]^{\frac{1}{2}}.$$

\begin{align*}
 I_1
 \geq&
 \frac{Q-2}{4}\epsilon_1(r_1,r_1)+\frac{Q-2}{4}\sigma_0 \left[(r_2, r_2)+(r_3,r_3)\right]
 \\
&\hspace*{2.6cm} -\frac{Q-2}{4}(1 + \epsilon_0+2\epsilon_1)\frac{\sqrt{2\,\tilde\delta}}\gamma (r_1,r_1)^{\frac{1}{2}}\left[(r_2, r_2)+(r_3,r_3)\right]^{\frac{1}{2}}
 \\
&\hspace*{2.6cm}
-\frac{Q-2}{2}( \epsilon_0+  \epsilon_1)\frac{2\sqrt{3}\sqrt{2\tilde\delta}}\gamma
(r_1,r_1)^{\frac{1}{2}}\left[(r_2, r_2)+(r_3,r_3)\right]^{\frac{1}{2}}\\
\geq &\left(\frac{Q-2}{2}\sqrt{\epsilon_1\sigma_0}-\frac{Q-2}{4}(1+(2+4\sqrt{3})(\epsilon_0+ \epsilon_1))\frac{\sqrt{2\tilde\delta}}\gamma \right)(r_1,r_1)^{\frac{1}{2}}\left[(r_2, r_2)+(r_3,r_3)\right]^{\frac{1}{2}}\geq 0
\end{align*}
if any $\epsilon_0< \frac{1}{3} $ and $\sigma_0$ satisfies $\sigma_0\ge\overline\sigma_0(\gamma,\epsilon_0,\tilde\delta)$ with
\be{sigma0}
\overline\sigma_0(\gamma,\epsilon_0,\delta):=\frac{1}{2\epsilon_1}(1+(2+4\sqrt{3})(\epsilon_0+ \epsilon_1))^2 \frac{\delta}{\gamma^2}.
\ee
Then we finish the proof of $I_1\ge0$.
\end{proof}



\subsection{Bound on \texorpdfstring{$I_3$}{I3}}

The idea for bounding this term is to use the Sobolev inequality. The extra coefficient $2\leq q\leq3$ gives us enough room to accomodate all error terms.
\begin{proposition}\label{Prop:estI2} Assume that $\tilde\delta\in(0,1)$ and $0<\epsilon_0<\frac13$. With
\be{epsilon2}
\epsilon_2 := \frac14\,(1-3\,\epsilon_0) {\rm \ \ and} \ \    \sigma_0=\frac2q\,\epsilon_2 ,
\ee
 for all $n\geq 2$, all $\tilde\delta\leq1$ and all $r$ as in Theorem~\ref{unifboundclose}, one has
$$
I_3\geq 0\,.
$$
\end{proposition}
\begin{proof} Taking into account the choice for $\sigma_0$, we have
$$
I_3=(1-\theta\,\epsilon_0)  \mathcal{E}[r_3] - \tfrac 2q (\frac{Q-2}{4})^2\((1+\epsilon_2\,\theta) \int_{\mathbb{S}^{2n+1}}{|r_3|^q}d\xi + \epsilon_2\,\theta (r_3,r_3)\)$$
We have $\|r_3\|^q_{L^q({\mathbb{S}^{2n+1}})} \leq\|r_3\|^2_{L^q({\mathbb{S}^{2n+1}})} $
 and $(r_3,r_3)\leq \|r_3\|^2_{L^q({\mathbb{S}^{2n+1}})}$ by H\"older's inequality. Thus,  we obtain
\begin{align*}
I_3 & \geq (1-\theta\,\epsilon_0) \mathcal{E} [r_3]  -\tfrac2q(\frac{Q-2}{4})^2 (1+2\,\epsilon_2\,\theta) \left( \int_{\mathbb{S}^{2n+1}}{|r_3|^q}d\xi \right)^{2/q}\\
&\geq(1-\theta\,\epsilon_0-\tfrac2q(1+2\,\epsilon_2\,\theta))\mathcal{E} [r_3] \\
& \geq \frac\theta q\,\(1-q\,\epsilon_0-4\,\epsilon_2\) \mathcal{E}[r_3]\ge0.
\end{align*}
where
using  Sobolev inequality \eqref{eqsobq} in the second inequality.
\end{proof}


\subsection{Bound on \texorpdfstring{$I_2$}{I2}}

At this point in the proof, for given $0<\epsilon_0<\tfrac13$, we have fixed the parameters $\epsilon_1$ and $\epsilon_2$ and we have found a $\delta_3$ such that $I_1$, $I_3\geq 0$ under the assumption $\tilde\delta\leq\delta_3$. Here we show that, by further decreasing $\tilde\delta$ if necessary, we can ensure that $I_3\geq 0$. The idea to achieve this is to use that $r_2$ satisfies an improved spectral gap inequality.
\begin{proposition}\label{Prop:estI3}
For any $0<\epsilon_0<\frac13$, let $\sigma_0=\frac2q\,\epsilon_2$. Then there is a $\delta_2\in(0,1)$ such that, for all $n\geq 2$, all $\tilde\delta\leq\delta_2$ and all~$r$ as in Theorem~\ref{unifboundclose}, one has
$$
I_2\geq 0\,.
$$
\end{proposition}
\begin{proof}
Since $Y_{k,j}\in H_{k,j}$ can also be seen as the $k+j$-order spherical harmonics of real $(2n+1)$ dimensional sphere $S^{2n+1}$,  according to~\cite[Theorem~1]{Duoandikoetxea}, for any $p\in[2,\infty)$ and $L^2$-normalized $Y_{k,j}$, there holds
$$
\|Y_{k,j}\|_{L^p({\mathbb{S}^{2n+1}})}
 \leq (p-1)^{\frac{k+j}{2}}\,.
$$
Thus, we can bound
$$
|(Y_{k,j},r_2)| \leq \|Y_{k,j}\|_{L^4({\mathbb{S}^{2n+1}})}\mu\big(\{ r_2 >0\}\big)^\frac14 (r_2,r_2)^{\frac{1}{2}}\leq 3^\frac k2\,\mu\big(\{ r_2 >0\}\big)^\frac14(r_2,r_2)^{\frac{1}{2}}.
$$
Meanwhile,
$$
\mu\big(\{ r_2 >0\}\big) = \mu\big(\{r>\gamma\}\big) \leq \frac1{\gamma^q} \|r\|_{L^q({\mathbb{S}^{2n+1}})} ^q \leq \frac{\tilde\delta^{q/2}}{\gamma^q}\,.
$$
Then
 for any $\mathrm L^2$-normalized spherical harmonic $Y_{k,j}$ of degree $j+k\in\N$, we have
\begin{equation}
\label{eq:projectionr2}
|(Y_{k,j},r_2)|  \leq 3^\frac k2\,\gamma^{-\frac q4}\,\tilde\delta^{\frac q8} (r_2,r_2)^{\frac{1}{2}}.
\end{equation}
If $\pi_{k,j}r_2$ denotes the projection of $r_2$ onto spherical harmonics of degree~$k+j$, namely, if we set $r_2=\sum \limits_{k+j\geq 2}Y_{k,j}$ and $Y_{k,j}=\frac{\pi_{k,j}r_2}{(\pi_{k,j}r_2,\pi_{k,j}r_2)^{\frac{1}{2}}}
$,  it follows
from \eqref{eq:projectionr2} we  get that
$$
(\pi_{k,j}r_2,\pi_{k,j}r_2)^{\frac{1}{2}} \leq 3^\frac k2\,\gamma^{-\frac q4}\,\tilde\delta^\frac q8 (r_2,r_2)^{\frac{1}{2}}.
$$
Next, for any $K,J\in\N$, if $\Pi_{K,J}\,r_2 := \sum \limits_{k<K,j<J} \pi_{k,j}\,r_2$ denotes the projection of $r_2$ onto spherical harmonics of degree less than $J+K$, then
\begin{align*}
(\Pi_{K,J}r_2,\Pi_{K,J}r_2)^{\frac{1}{2}}&= \Big( {\scriptstyle\sum\limits_{k<K,j<J}}(\pi_{k,j}r_2,\pi_{k,j}r_2)\Big)^{\!1/2} \\
&\leq \gamma^{-\frac q4}\,\tilde\delta^\frac q8 (r_2,r_2)^{\frac{1}{2}}\sqrt{{\scriptstyle\sum\limits_{k<K,j<J}} 3^{k+j}} \leq 3^\frac {K+J}2\,\gamma^{-\frac q4}\,\tilde\delta^\frac q8 (r_2,r_2)^{\frac{1}{2}}.
\end{align*}
{}From this we conclude that
\begin{align*}
\mathcal{E}_0[r_2] & \geq \mathcal{E}_0[r_2-\Pi_{K,J} r_2] \\
& \geq\left( \frac{Q-2}{4}(K+J)+KJ \right)
\((1-\Pi_{K,J})r_2,(1-\Pi_{K,J})r_2\)
\\
&\geq  \frac{Q-2}{4}(K+J) \left( (r_2,r_2) -(\Pi_{K,J}r_2,\Pi_{K,J}r_2)\right)\\
& \geq \frac{Q-2}{4}(K+J)\left( 1- 3^{K+j} \gamma^{-\frac q2}\,\tilde\delta^\frac q4 \right)(r_2,r_2).
\end{align*}
Consequently,
$$
I_2 \geq \left( (1-\theta \epsilon_0)\frac{Q-2}{4}(K+J)\left( 1- 3^{K+J}\gamma^{-\frac q2}\,\tilde\delta^\frac q4 \right) - Q \left( 1+ \epsilon_0 + \sigma_0 + C_{\epsilon_1,\epsilon_2} \right) \right)(r_2,r_2).
$$
 If $\tilde\delta\le \delta_2:=\frac12\frac{\gamma^2}{3^{2(K+J)}}$, we have $1- 3^{K+J}\,\gamma^{-\frac q2}\,\tilde\delta^\frac q4\ge\frac12$. Let we choose $L:=K+J\in\N$  such that
\be{L-delta2}
L:= 4 \frac{1+ \epsilon_0 + \sigma_0 + C_{\epsilon_1,\epsilon_2}}{1-\epsilon_0}
\ee
where $\sigma_0$ is given by~\eqref{epsilon2},  we can conclude  $I_2\ge0$.
\end{proof}

Fix some $\epsilon_0\in(0,\frac 13)$ and with the choice
$$
\gamma=\epsilon_2=2\,\epsilon_1=\tfrac14\,(1-3\,\epsilon_0)\quad\mbox{and}\quad\sigma_0=\frac2q\,\epsilon_2,
$$
   and an arbitrary choice of
$$
M\ge2\,\gamma\,,\quad\overline M\ge\sqrt e\quad\mbox{and}\quad\epsilon=\gamma
$$
which determines $C_{\epsilon_1,\epsilon_2}=C_{\gamma,\epsilon,M}$ according to~\eqref{epsilon1-epsilon2-C} on the other hand, the condition
$$
\tilde\delta=\min\big\{\delta_1,\delta_2\big\}
$$
with $\delta_1:=\frac{4\epsilon_1\epsilon_2\gamma^2}{q(1+(2+4\sqrt{3})(\epsilon_0+\epsilon_1))^2}$, we claim that $I_1$, $I_2$ and $I_3$ are nonnegative, which completes the proof of Lemma~\ref{unifboundclose} as $n\geq 2$.

\section{The stability of HLS inequalities by Carlen's dual stability method}

In this section, we will prove the stability of Hardy-Littlewood-Sobolev inequalities on the Heisenberg group when $\lambda=Q-2$ (Theorem \ref{HLS stability})  from the stability of Sobolev inequalities and the dual stability
method introduced by Carlen \cite{Car1}, \cite{Car2}.

For any function $g$ on $\mathbb{S}^{2n+1}$, define the function $f$ on $\mathbb{H}^n$ by $f(u)=|J_{\mathcal{C}}|^{\frac{Q+2}{2Q}}g\circ \mathcal{C}$, where $\mathcal{C}$ be the Cayley transform from $\mathbb{H}^n$ to $\mathbb{S}^{2n+1}$ and $J_{\mathcal{C}}$ be the Jacobian of this transform. Then
$$\|g\|_{L^{\frac{2Q}{Q+2}}(\mathbb{S}^{2n+1})}=\|f\|_{L^{\frac{2Q}{Q+2}}(\mathbb{H}^{n})},$$
and
$$\iint_{\mathbb{H}^n\times \mathbb{H}^n}\frac{f(u)f(v)}{|u^{-1}v|^{Q-2}}du dv=2^{-n(Q-2)/Q}\iint_{\mathbb{S}^{2n+1}\times \mathbb{S}^{2n+1}}\frac{g(\xi)g(\eta)}{|1-\xi\cdot\bar{\eta}|^{\frac{Q-2}{2}}}d\xi d\eta.$$
Then sharp Hardy-Littlewood-Sobolev inequalities on $\mathbb{S}^{2n+1}$ when $\lambda=Q-2$ states (see Theorem 2.2 in \cite{Frank-Lieb})
\begin{equation*}
|\mathbb{S}^{2n+1}|^{(Q-2)/Q}\frac{n!}{\Gamma^2(\frac{Q+2}{4})}\|g\|^2_{L^{\frac{2Q}{Q+2}}(\mathbb{S}^{2n+1})}\geq\iint_{\mathbb{S}^{2n+1}\times \mathbb{S}^{2n+1}}\frac{g(\xi)g(\eta)}{|1-\xi\cdot\bar{\eta}|^{\frac{Q-2}{2}}}d\xi d\eta,
\end{equation*}
with equality if and only if
$$g=\frac{c}{|1-\overline{\eta}\cdot\zeta|^{(Q+2)/2}}$$
for  some $c\in \mathbb{C}$, $\eta=(\eta_1,...,\eta_{n+1})\in \mathbb{C}^{n+1}$ with $|\eta| <1$.

Then we only need to prove the stability of HLS inequalities on the CR sphere $\mathbb{S}^{2n+1}$. In fact, we will prove
\begin{align}\label{hls sta}
& \|g\|^2_{L^{\frac{2Q}{Q+2}}(\mathbb{S}^{2n+1})}-|\mathbb{S}^{2n+1}|^{(2-Q)/Q}\frac{\Gamma^2(\frac{Q+2}{4})}{n!}\iint_{\mathbb{S}^{2n+1}\times \mathbb{S}^{2n+1}}\frac{g(\zeta)g(\eta)}{|1-\zeta\cdot\bar{\eta}|^{(Q-2)/2}}d\zeta d\eta\nonumber\\
& \geq \min{\left\{\frac{4\beta(Q+2)}{Q(Q-2)},1\right\}}\frac{Q+2}{2(Q-2)}\inf_{h\in \mathcal{M}^{\ast}_{HLS}}\|g-h\|^2_{\frac{2Q}{Q+2}},
\end{align}
where $$\mathcal{M}^{\ast}_{HLS}=\left\{\frac{c}{|1-\overline{\eta}\cdot\xi|^{2/q}}:~~c\in \mathbb{R},~~\eta=(\eta_1,...,\eta_{n+1})\in \mathbb{C}^{n+1}\ \text{with}~~|\eta| <1\right\}.$$ Since Carlen's dual stability method is based on the Legendre transform, we first explain the sharp HLS inequality can be deduced by the sharp Sobolev inequality and vice-versa.
Define the operator $A_{2}$ in $L^2(\mathbb{S}^{2n+1})$, which act as multiplication on $\mathcal{H}_{j,k}$ by
$$\lambda_{j,k}= \frac{\Gamma(j+\frac{Q+2}{4})\Gamma(k+\frac{Q+2}{4})}{\Gamma(j+\frac{Q-2}{4})\Gamma(k+\frac{Q-2}{4})}= (\frac{Q-2}{4}+k)(\frac{Q-2}{4}+j),$$
then we can write the sharp Sobolev inequality on $\mathbb{S}^{2n+1}$ by
$$(u,\mathcal{L}u)\geq |\mathbb{S}^{2n+1}|^{\frac{2}{Q}}(\frac{Q-2}{4})^2\left(\int_{\mathbb{S}^{2n+1}}|u|^{\frac{2Q}{Q-2}}d\xi\right)^{\frac{Q-2}{Q}}.$$
Since the eigenvalue of the operator with kernel $|1-\zeta\cdot\bar{\eta}|^{-2\alpha}$($-1<\alpha<(n+1)/2$) on the subspace $\mathcal{H}_{j,k}$ is (see \cite{Frank-Lieb})
$$E_{j,k}=\frac{2\pi^{n+1}\Gamma(n+1-2\alpha)\Gamma(j+\alpha)\Gamma(k+\alpha)}{\Gamma^2(\alpha)\Gamma(j+n+1-\alpha)\Gamma(k+n+1-\alpha)},$$
 it is easy to check that the inverse operator $\mathcal{L}^{-1}$ of $\mathcal{L}$ can be written as
$$\mathcal{L}^{-1}(v)(\zeta)=\frac{\Gamma^2(\frac{Q-2}{4})}{2\pi^{n+1}}\int_{\mathbb{S}^{2n+1}}\frac{v(\eta)}{|1-\zeta\cdot\bar{\eta}|^{\frac{Q-2}{2}}}d\eta.$$
Recall that we consider the real-valued functions and $2^{\ast}=\frac{2Q}{Q-2}$. Denote
$$\langle u,v\rangle : =2\int_{\mathbb{S}^{2n+1}}u(\eta)v(\eta)d\eta,$$
and
$$\mathcal{I}(u) = |\mathbb{S}^{2n+1}|^{\frac{2}{Q}}(\frac{Q-2}{4})^2\|u\|^2_{L^{\frac{2Q}{Q-2}}(\mathbb{S}^{2n+1})}.$$
The Legendre transform $\mathcal{I}^\ast$ of $\mathcal{I}$ is defined by
$$\mathcal{I}^\ast(v)=\sup_{u\in L^{2^{\ast}}(\mathbb{S}^{2n+1})}\left\{\langle u,v\rangle-\mathcal{I}(u)\right\}.$$
It is obviously
$$\mathcal{I}^\ast(v)= |\mathbb{S}^{2n+1}|^{-\frac{2}{Q}}(\frac{Q-2}{4})^{-2}\|v\|^2_{L^{\frac{2Q}{Q+2}}(\mathbb{S}^{2n+1})}.$$
On the other hand, let $\mathcal{F}(u)=(u,\mathcal{L}u)$ and the Legendre transform
$\mathcal{F}^\ast$ of $\mathcal{F}: S^{1}(\mathbb{S}^{2n+1})\rightarrow [0,+\infty)$ defined on $S^{-1}(\mathbb{S}^{2n+1})$ is given by
$$\mathcal{F}^\ast(v)=\sup_{u\in S^{1}(\mathbb{S}^{2n+1})}\left\{\langle u,v\rangle-\mathcal{F}(u)\right\}.$$
A simple calculation gives
$$\mathcal{F}^\ast(v)=(v,\mathcal{L}^{-1}v)=\frac{\Gamma^2(\frac{Q-2}{4})}{2\pi^{n+1}}\iint_{\mathbb{S}^{2n+1}\times \mathbb{S}^{2n+1}}\frac{v(\eta)v(\xi)}{|1-\xi\cdot\bar{\eta}|^{\frac{Q-2}{2}}}d\eta d\xi.$$
Using the fact $|\mathbb{S}^{2n+1}|=\frac{2\pi^{n+1}}{n!}$ and $|\mathbb{S}^{2n+1}|^{-\frac{2}{Q}}(\frac{Q-2}{4})^{-2}\times (\frac{\Gamma^2(\frac{Q-2}{4})}{2\pi^{n+1}})^{-1} =|\mathbb{S}^{2n+1}|^{\frac{Q-2}{Q}}\frac{n!}{\Gamma^2(\frac{Q+2}{4})}$, we can obtain the dual relationship between sharp HLS and Sobolev inequality on $\mathbb{S}^{2n+1}$ by Legendre transform.

Now by the stability of Sobolev inequalities on $\mathbb{S}^{2n+1}$ (Theorem \ref{sob sta sph}) and the Sobolev inequality itself we have
\begin{align*}
& \mathcal{F}(u)-\mathcal{I}(u)=(u,\mathcal{L}u)-|\mathbb{S}^{2n+1}|^{\frac{2}{Q}}(\frac{Q-2}{4})^2\left(\int_{\mathbb{S}^{2n+1}}|u|^{\frac{2Q}{Q-2}}d\xi \right)^{\frac{Q-2}{Q}}\\
& \geq \frac{\beta}{Q}\inf_{h\in \mathcal{M}_{\ast}}(u-h,\mathcal{L}(u-h))\\
& \geq \frac{\beta}{Q}|\mathbb{S}^{2n+1}|^{\frac{2}{Q}}(\frac{Q-2}{4})^2\inf_{h\in \mathcal{M}_{\ast}}\|u-h\|^2_{\frac{2Q}{Q-2}}.
\end{align*}
Since $\mathcal{I}^\ast$ is $\frac{Q+2}{Q-2}$ convex follow from the $(p-1)$-convexity of $v\rightarrow \|v\|^2_p$ when $1<p<2$(see \cite{Car1} for details), then by Theorem 2.4 in
\cite{Car2} we have
\begin{align*}
& |\mathbb{S}^{2n+1}|^{-\frac{2}{Q}}(\frac{Q-2}{4})^{-2}\|v\|^2_{L^{\frac{2Q}{Q+2}}(\mathbb{S}^{2n+1})}-\frac{\Gamma^2(\frac{Q-2}{4})}{2\pi^{n+1}}\iint_{\mathbb{S}^{2n+1}\times \mathbb{S}^{2n+1}}\frac{v(\eta)v(\xi)}{|1-\xi\cdot\bar{\eta}|^{\frac{Q-2}{2}}}d\eta d\xi\\
& =\mathcal{I}^{\ast}(v)-\mathcal{F}^{\ast}(v)\geq \min{\left\{\frac{4\beta(Q+2)}{Q(Q-2)},1\right\}}\frac{Q+2}{2(Q-2)}|\mathbb{S}^{2n+1}|^{-\frac{2}{Q}}(\frac{Q-2}{4})^{-2}\inf_{h\in \mathcal{M}^{\ast}_{HLS}}\|v-h\|^2_{\frac{2Q}{Q+2}}.
\end{align*}
That means
\begin{align*}
& \|v\|^2_{L^{\frac{2Q}{Q+2}}(\mathbb{S}^{2n+1})}-|\mathbb{S}^{2n+1}|^{\frac{2-Q}{Q}}\frac{\Gamma^2(\frac{Q+2}{4})}{n!}\iint_{\mathbb{S}^{2n+1}\times \mathbb{S}^{2n+1}}\frac{v(\eta)v(\xi)}{|1-\xi\cdot\bar{\eta}|^{\frac{Q-2}{2}}}d\xi d\eta\\
& \geq \min{\left\{\frac{4\beta(Q+2)}{Q(Q-2)},1\right\}}\frac{Q+2}{2(Q-2)}\inf_{h\in \mathcal{M}^{\ast}_{HLS}}\|v-h\|^2_{\frac{2Q}{Q+2}},
\end{align*}
which is (\ref{hls sta}) and thus this completes the proof of Theorem \ref{HLS stability}.

\bigskip\begin{center}\rule{2cm}{0.5pt}\end{center}\bigskip
\end{document}